\documentclass[11pt]{amsart}

\usepackage{amssymb, amsthm,amstext,amsfonts,amsmath,graphics,lipsum}

\usepackage[all,2cell]{xy} \UseAllTwocells \SilentMatrices
\usepackage{xypic}

\usepackage{graphicx}

\usepackage[margin=1.2in]{geometry}

\usepackage[latin1]{inputenc}
\usepackage[all]{xy}

\usepackage{tikz}
\usetikzlibrary{matrix}

\makeatletter
\newtheorem*{rep@theorem}{\rep@title}
\newcommand{\newreptheorem}[2]{%
\newenvironment{rep#1}[1]{%
 \def\rep@title{#2 \ref{##1}}%
 \begin{rep@theorem}}%
 {\end{rep@theorem}}}
\makeatother

\newtheorem{theorem}{Theorem}
\newtheorem*{theorem*}{Theorem}
\newreptheorem{theorem}{Theorem}
\numberwithin{theorem}{section}
\newtheorem{theorem-definition}[theorem]{Theorem-Definition}
\newtheorem{proposition}[theorem]{Proposition}

\newtheorem*{acknowledgements*}{Acknowledgements}

\newtheorem{corollary}[theorem]{Corollary}
\newtheorem*{corollary*}{Corollary}
\newreptheorem{corollary}{Corollary}
\newtheorem{lemma}[theorem]{Lemma}
\newtheorem*{lemma*}{Lemma}
\newtheorem{conjecture}[theorem]{Conjecture}

\theoremstyle{definition}
\newtheorem{definition}[theorem]{Definition}
\newtheorem*{definition*}{Definition}

\newtheorem{claim}[theorem]{Claim}
\newtheorem{note}[theorem]{Remark}
\newtheorem*{note*}{Remark}

\theoremstyle{definition} 
\theoremstyle{remark}
\numberwithin{equation}{section}
\DeclareMathOperator{\Out}{Out}

\DeclareMathOperator{\Vol}{Vol}

\DeclareMathOperator{\Ker}{Ker}
\DeclareMathOperator{\Diff}{Diff}
\DeclareMathOperator{\PSL}{PSL}

\newcommand{\fun}[3]{#1 \colon #2 \to #3}

\makeatletter
\newcommand{\tpitchfork}{%
  \vbox{
    \baselineskip\z@skip
    \lineskip-.52ex
    \lineskiplimit\maxdimen
    \m@th
    \ialign{##\crcr\hidewidth\smash{$-$}\hidewidth\crcr$\pitchfork$\crcr}
  }%
}
\makeatother

\begin{document}

\title[A lower bound for the volumes of complements of periodic geodesics]{A lower bound for the volumes of complements of periodic geodesics}
\author[J. A. Rodriguez-Migueles]{JOSE ANDRES RODRIGUEZ-MIGUELES}
\begin{abstract}
Every closed geodesic $\gamma$ on a surface has a canonically associated knot $\widehat\gamma$ in the projective unit tangent bundle. We study, for $\gamma$ filling, the volume of the associated knot complement with respect to its unique complete hyperbolic metric. We provide a lower bound for the volume relative to the number of homotopy classes of $\gamma$-arcs in each pair of pants of a pants decomposition of the surface.
\end{abstract}
\maketitle
\section{Introduction}

Let $\Sigma$ be a complete, orientable hyperbolic surface of finite area. Every closed geodesic $\gamma$ in $\Sigma$ has a canonical lift $\widehat\gamma$ in the projective unit tangent bundle  $PT^1\Sigma,$ namely the image under the map ${T^1\Sigma}\rightarrow{PT^1\Sigma}$  of the corresponding periodic orbit of the geodesic flow. Such canonical lift  $\widehat\gamma$ is an embedding of $\mathbb{S}^1$ into $PT^1\Sigma$ and can then be considered as a knot in $PT^1\Sigma.$ 
\vskip .2cm
More generally, given a closed geodesic in a hyperbolic $2$-orbifold one can consider its canonical lift as a knot in the corresponding unit tangent bundle. This is possibly particularly interesting in the case of the modular surface $\Sigma_{mod}=\mathbb{H}^2/\PSL_2(\mathbb{Z})$, because its unit tangent bundle is homeomorphic to the complement of the trefoil knot in $\mathbb{S}^3.$ Therefore, canonical lifts of closed geodesics on the modular surface can be seen as knots in $\mathbb{S}^3.$ In fact, in \cite{Ghy07} Ghys observed that the obtained knots are Lorentz knots. For facts about such knots \cite{BWS83,BK09,Deh11}.
\vskip .2cm
The aim of this paper is to study $M_{\widehat\gamma},$ the complement of a regular neighborhood of the canonical lift $\widehat\gamma$ in $PT^1\Sigma,$ for a general hyperbolic surface $\Sigma.$  Foulon and Hasselblatt \cite{FH13}  proved, with the advising of Calegari, Fenley, Otal and Walsh, that $M_{\widehat\gamma}$ admits a complete hyperbolic metric of finite volume as soon as $\gamma$ fills the surface. Such metric is unique by the Mostow's Rigidity Theorem, meaning that any geometric invariant is a topological invariant.  We will be interested in estimating the volume of $M_{\widehat\gamma},$ when $\gamma$ is a filling closed geodesic.
\vskip .2cm
Our main result is to give a lower bound on the volume of the complement $M_{\widehat\gamma}$ (Theorem \ref{1}). We start however by commenting on upper bounds for the volume of $M_{\widehat\gamma}.$ Bergeron, Pinsky and Silberman have already studied in \cite{BPS16}  the problem of finding an upper bound, by giving one which is linear in the length of the geodesic. Nevertheless, it is easy to construct sequences of closed geodesics with length approaching to infinity but whose associated canonical lift complements have uniformly bounded volume. For example, suppose that $\gamma$ is filling and $\phi$ is a diffeomorphism of the surface representing an infinite order element in the mapping class group: the length of the geodesics in the sequence $\{\phi^n(\gamma)\}_{n\in\mathbb{N}}$ tend to infinity as $n$ grows, but every $M_{\widehat{\phi^n(\gamma)} }$ is homeomorphic to $M_{\widehat\gamma}.$ There are however more interesting examples:

\begin{theorem}\label{2}
Given a hyperbolic surface $\Sigma,$ there exist a constant $V_0>0$ and a sequence $\{\gamma_n\}$ of filling closed geodesic on $\Sigma,$ with $M_{\widehat\gamma_n}\not\cong M_{\widehat\gamma_k}$ for every $k\not= n,$ such that $\Vol(M_{\widehat\gamma_n})< V_0$ for every $n\in\mathbb{N}.$  Moreover, for any sequence $\{X_n\}$ of hyperbolic metrics on $\Sigma,$  we have that $\ell_{X_n}(\gamma_n) \nearrow \infty.$
 \end{theorem}

Bergeron, Pinsky and Silberman gave already in (\cite{BPS16}, Example 3.1) exam\-ples satisfying Theorem \ref{2} in the case of the modular surface. Their examples and the ones in the proof of Theorem \ref{2} are different. To explain why this is the case we recall that in the modular surface (\cite{BPS16}, Section 3) they found an upper bound which is proportional to the period plus the sum of the logarithms of the coefficients corresponding to the geodesic's continued fraction expansion. The examples provided in  (\cite{BPS16}, Example 3.1) have unbounded length but bounded period. On the other hand, the examples used to prove Theorem \ref{2} have unbounded period. This yields:

\begin{corollary}\label{mod1}
For the modular surface $\Sigma_{mod},$ there exist a constant $V_0>0$ and a sequence $\{\gamma_k\}$ of filling closed geodesics on $\Sigma_{mod},$  with $M_{\widehat\gamma_n}\not\cong M_{\widehat\gamma_k}$ for every $k\not= n,$ such that $\Vol(M_{\widehat\gamma_k})< V_0$ for every $k\in\mathbb{N}$ and the period of the continued fraction expansion of $\gamma_k$ tends to infinity.
\end{corollary}

It is perhaps worth mentioning that there is a numerical example for Corollary \ref{mod1} in (\cite{BPS17}, Example 5.2). Even more interesting is that in their paper the authors presented numerical evidence of geodesics on the modular surface for which the volume's growth is linear in the geometric length of the geodesics. We show that, up to a logarithmic factor, such geodesics do actually exist in any hyperbolic surface: 

\begin{theorem}\label{pi-b-2}
Given a hyperbolic metric $X$ on a surface $\Sigma,$ there exist a sequence $\{\gamma_n\}$ of filling closed geodesics on $\Sigma$ with $\ell_X(\gamma_n)\nearrow \infty$ such that,
$$ \Vol(M_{\widehat{\gamma_n}})\geq c_X\frac{ \ell_X(\gamma_n)}{\ln(\ell_X(\gamma_n))},$$
where $c_X$ depends on the hyperbolic metric $X.$
 \end{theorem}
 
Returning to the case of the modular surface, we can reformulate Theorem \ref{pi-b-2} in terms of the period of the geodesic's continued fraction expansion, as follows:
\begin{theorem}\label{mod}
For the modular surface $\Sigma_{mod},$ there exist a sequence $\{\gamma_k\}$ of filling closed geodesics on $\Sigma_{mod}$ with $n_{\gamma_k}\nearrow \infty$ such that,
$$ \Vol(M_{\widehat{\gamma_k}})\geq v_3 \frac{n_{\gamma_k}}{12},$$
where $n_{\gamma_k}$ is half the period of the continued fraction expansion of $\gamma_k$ and $v_3$ is the volume of a regular ideal tetrahedra. 
\end{theorem}

The reader might wonder why we put great emphasis on the period of the geodesics on the modular surface. The reason for this is that in a further research we intend to pursue an upper bound for the volume of the canonical lift complement which is linear in the period of the geodesic's continued fraction expansion.
\vskip .2cm
We point out that the main ingredient to achieve the proofs of Theorem \ref{pi-b-2} and Theorem \ref{mod}, is our next result which estimates a lower bound for the volume of the canonical lift complement in terms of combinatorial data between the geodesic and a given pants decomposition on $\Sigma.$ 

\begin{theorem} \label{1}
Given a pants decomposition $\Pi$ on a hyperbolic surface $\Sigma,$  and $\gamma$ a filling closed geodesic, we have that:
 $$\Vol(M_{\widehat\gamma})\geq \frac{v_3}{2}\sum_{P \in \Pi}(\sharp\{\mbox{homotopy classes of} \hspace{.2cm}  \gamma\mbox{-arcs in} \hspace{.2cm} P\}-3),$$ 
  where $v_3$ is the volume of a regular ideal tetrahedra.
  \end{theorem}
  
It is well known in $3$-dimensional topology that giving a lower bound for the volume of a complete hyperbolic $3$-manifolds is a more difficult problem than the upper bound, so it will be not surprising that Theorem \ref{1} uses a more profound result (\cite{AST07}, Theorem 9.1) due to Agol, Storm and Thurston. This result allows us to give a lower bound in terms of the simplicial volume of the manifold constructed by doubling the pieces of $M_{\widehat\gamma}$ that result from cutting it along the pre-image under the map $PT^1\Sigma\rightarrow\Sigma$ of the pants curves.
\vskip .2cm
Before presenting the structure of this paper, let us go back to the case of the modular surface. It will be interesting to calculate, by using Theorem \ref{1}, a lower bound for the geodesics in \cite{BPS17}, which come from the ideal class group of the fields $\mathbb{Q}(\sqrt{d})$ with $d$ a square-free positive integer bigger than $1.$ The interest of these closed geodesics is that they are uniformly distributed on $PT^1\Sigma_{mod}$ \cite{Duk88}.
\vskip .2cm
This paper is organized as follows. After the introduction, we begin section \ref{s1} by presenting some well known facts on transversal homotopies, one of the main technical tools that allows us to deform the geodesics without changing the topological type of their canonical lift complement.  We also recall some $3$-manifold properties of $M_{\widehat\gamma},$ such as its JSJ-decomposition, which rely on the result of the hyperbolicity of $M_{\widehat\gamma}$ (\cite{FH13}, Theorem 1.12). The only new result found in section \ref{s1} is the following on the topological type of $M_{\widehat\gamma}$, which maybe is interesting in its own right:

 \begin{theorem}\label{charac}
Let $\gamma$ and $\eta$ be two closed geodesics, where the class of $\gamma$ in $H_1(\Sigma; \mathbb{R})$ is not trivial. Then $M_{\widehat\gamma}$  is homeomorphic to $M_{\widehat\eta}$ if and only if there is a diffeomorphism $\phi$ of $\Sigma$ such that $\phi(\gamma)$ is homotopic to $\eta$. \end{theorem}

Section \ref{s2} is devoted to construct sequences of geodesics whose canonical lift complements are not homeomorphic and their volumes are universally bounded (Theorem \ref{2}). This sequence of canonical lifts is obtained by performing annular Dehn fillings on a link in $PT^1\Sigma.$ It follows by using transveral homotopies that the length of the geodesic in the sequence is unbounded for any sequence of hyperbolic metrics. Specifically, one shows that the self-intersection number of the sequence is unbounded.
\vskip .2cm
Later, in section \ref{s3} we prove our main result (Theorem \ref{1}) that gives a combinatorial lower bound on the volume of $M_{\widehat\gamma}.$ This allows to construct sequences of geodesics whose volume grows as the geodesic complexity increases, in terms of the length (Theorem \ref{pi-b-2}) or in terms of the period (Theorem \ref{mod}). 
\vskip .2cm
We conclude in section \ref{s4}, by discussing what happens to the volume's lower bound on the lift complement if one changes the canonical lift to a non-canonical lift in $PT^1\Sigma$ over the same filling closed geodesic in the surface.

\begin{acknowledgements*}
\textnormal{ I gratefully thank Juan Souto for suggesting to study this invariant, his useful advice and insight. I was also greatly benefited with conversations from many people during the course of this work, particularly Tommaso Cremaschi, Pierre Dehornoy,  Mario Eudave Mu\~noz, Anna Lenzhen, Max Neumann Coto and Joan Porti. The author thanks the Centre Henri Lebesgue ANR-11-LABX-0020-01 for creating an attractive mathematical environment. Moreover, I would like to thank the anonymous referee for many helpful comments and suggestions.}
\end{acknowledgements*}

\section{Topology and geometry of $M_{\widehat\gamma}$}\label{s1}

In this section we start by noticing that the canonical lift can be defined for a larger set of closed curves on surfaces and recall a special type of homotopies between them which preserve the topological type of their canonical lift complement. Later, in Theorem \ref{charac} we prove that for any non null-homologous closed geodesic $\gamma,$ the topological type of $M_{\widehat\gamma}$ is determined by its mapping class group orbit. Finally, we set some preliminaries on 3-dimensional manifolds in order to investigate the topological and geometrical features of $M_{\widehat\gamma},$ for example we recall why the filling condition on $\gamma$ implies that $M_{\widehat\gamma}$ admits a unique hyperbolic structure.

\subsection{Transversal homotopies}\label{th}

 Any immersed closed curve $\gamma$ on $\Sigma,$ has a canonical lift into  $PT^1\Sigma$ by considering\string: $$\widehat\gamma(s):=\left(\gamma(s),\frac{\dot \gamma(s)}{ \|\dot\gamma(s) \|}\right).$$ 
It happens to be an embedding if $\gamma$ is self-transverse, that is, an immersion without self-tangency points. Furthermore, we say that a homotopy $\fun{h}{I\times \mathbb{S}^1}{\Sigma},$ between two self-transverse closed curves $h_0$ and $h_1,$ is \textit{transversal}, if $h$ is smooth and $h_t$ is self-transverse for all $t\in I.$ 
 \vskip .2cm
  Notice that a transversal homotopy $h$ on $\Sigma,$ induces an isotopy between the respective canonical lifts $\widehat h_0$ and $\widehat h_1,$
  $$\fun{\widehat h}{ I\times\mathbb{S}^1}{PT^1\Sigma}
\hspace{.2cm}\mbox{where} \hspace{.2cm} \widehat h (t,s)=\left(h_t(s),\frac{\dot h_t(s)}{ \|\dot h_t(s) \|}\right).$$
By the isotopy extension lemma, there is an ambient isotopy $\fun{H}{I\times PT^1\Sigma}{PT^1\Sigma}$ between $\widehat h_0$ and $\widehat h_1,$ extending the isotopy $\widehat h,$ that satisfies\string:
$$H (t,\widehat h_0(s))=\widehat h_t(s).$$
It follows that the topological type of the canonical lift complement of a self-transverse closed curve is invariant under transversal homotopies. Recall now that in \cite{HS94} Hass and Scott proved that two homotopic self-transverse closed curves with minimal self-intersection, are transversally homotopic. Therefore, for any self-transverse closed curve with minimal self-intersection, its canonical lift complement is homeomorphic to the canonical lift complement of the unique closed geodesic associated to it.
 \vskip .2cm
Putting all this together we get the following fact, which we record here for later use:

\begin{lemma*}
The homeomorphism type of $M_{\widehat\gamma}$ is independent of the chosen hyperbolic metric on the hyperbolic surface.
\hfill $\square$
\end{lemma*}

In the constructions of examples around this article, it will be important to characterize when a closed curve has minimal self-intersection. In (\cite{HS85}, Theorem 2) Hass and Scott proved the following result which will help us to decide when a closed curve on a surface has this property\string:

\begin{theorem}[Hass-Scott]\label{2-gone}
Let $\alpha$ be a closed curve on a surface which has excess self-intersection. Then there is a singular $1$-gon or $2$-gon on the surface bounded by part of the image of $\alpha.$ \end{theorem}

 \begin{figure}[h]
\centering
\includegraphics[scale=0.25] {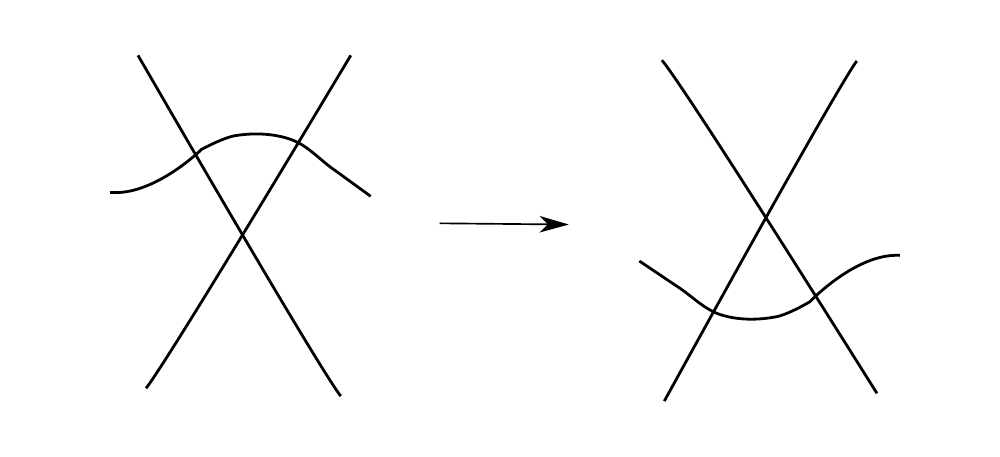}
\caption{A third Reidemeister move.
}\label{reide}
\end{figure}

A \textit{third Reidemeister move} on a closed curve is a local move which corresponds to pushing a branch of a curve across a self-intersection, as shown in Figure \ref{reide}, that is a special type of transversal homotopy. In (\cite{HS94}, Theorem 2.1, and \cite{GS97} , Theorem 1) Hass and Scott, and later De Graaf and Schrijver, proved that if two closed curves with minimal self-intersection are in the same homotopy class, then there is a sequence of third Reidemeister moves and an ambient isotopy on the surface between them. Consequently, we have the following result which might exist in the literature of curves on surfaces, but we were not able to find it in this form useful for our study.

\begin{corollary*} 
Let $\gamma$ and $\eta$ be two homotopic self-transverse closed curves with minimal self-intersection on $\Sigma,$ then there is a homeomorphism between the non simply connected components of $\Sigma\setminus\gamma$ and $\Sigma\setminus\eta.$\hfill $\square$\end{corollary*}

A closed curve $\alpha$ with minimal self-intersection on a surface $\Sigma$ is said to be \textit{filling} if $\Sigma\setminus\alpha$ is a collection of disks or once-punctured disks; equivalently, if its geometric intersection with respect to any essential simple closed curve is not zero. By the previous result, we conclude that this property is preserved by transversal isotopies.

\begin{corollary} \label{comp}
Given a self-transverse closed curve with minimal self-intersection, then the property of being filling over a surface is preserved by transversal homotopies.
\hfill $\square$
\end{corollary}

\subsection{Topological type of  $M_{\widehat\gamma}$}

In this subsection we prove that the canonical lift complement of a non null-homologous closed geodesic is determined by its mapping class group orbit. 

\begin{reptheorem}{charac}
Let $\gamma$ and $\eta$ be two closed geodesics, where the class of $\gamma$ in $H_1(\Sigma; \mathbb{R})$ is not trivial. Then $M_{\widehat\gamma}$  is homeomorphic to $M_{\widehat\eta}$ if and only if there is a diffeomorphism $\phi$ of $\Sigma$ such that $\phi(\gamma)$ is homotopic to $\eta$. \end{reptheorem}

\begin{proof}[\bf{Proof}]
By the discussion at the beginning of subsection \ref{th} and the fact that diffeomorphisms  preserve self-intersection number, we imply that closed geodesics in the same mapping class group orbit have homeomorphic canonical lift complements. This reduces the proof to one implication. 
\vskip .2cm
The first step is proving that every homeomorphism $\fun{f}{M_{\widehat\gamma}}{M_{\widehat\eta}},$ can be extended to a self-homeomorphism $\widehat f$ of $PT^1\Sigma$ with $\widehat f(\widehat\gamma)=\widehat\eta.$ Notice that if $f$ maps $m_{\widehat\gamma},$ the meridian of $\widehat\gamma,$ to the meridian of $\widehat\eta,$ then we can extend $f$ to the disks that they bound inside $PT^1\Sigma.$ Moreover, we can extend this homeomorphism along the core of $\mathcal{N}_{\widehat\gamma},$ the normal neigborhood of $\widehat\gamma,$ to $\mathcal{N}_{\widehat\eta}.$ As the morphism in homology induced by the restriction of $f$ on $\partial\mathcal{N}_{\widehat\gamma}$ is an isomorphism, then $H_1(f_{\mid\partial \mathcal{N}_{\widehat\gamma}})(\Ker(H_1(i_{\widehat\gamma})))=\Ker(H_1(i_{\widehat\eta})),$ where $i_{\widehat\gamma}$ is the inclusion of $\partial\mathcal{N}_{\widehat\gamma}$ into $ M_{\widehat\gamma},$ and analogous for $i_{\widehat\eta}.$ So we just need to prove the following claim\string:

\begin{claim}\label{mer}
If the class of $\gamma$ in $H_1(\Sigma; \mathbb{R})$ is not trivial, then $\Ker(H_1(i_{\widehat\gamma}))$ is generated by the class of $m_{\widehat\gamma}$ in $H_1(\partial \mathcal{N}_{\widehat\gamma}; \mathbb{R}).$
\end{claim}

Assuming Claim \ref{mer} we conclude the proof of Theorem \ref{charac}, by showing that $\widehat f$ induces a diffeomorphism on the surface $\Sigma.$ 
Using the fact that the kernel of the morphism  $\fun{\pi_*}{\pi_1(PT^1\Sigma)}{\pi_1(\Sigma)}$ is characteristic, we have the following morphism,
 
$$\Out(\pi_1(PT^1\Sigma))\rightarrow \Out(\pi_1(\Sigma)).$$

By the Baer-Dehn-Nielsen Theorem, there exist $\phi\in \Diff(\Sigma)$ associated to $\widehat f$ such that, 
$$\phi_*\circ \pi_*=\pi_*\circ  \widehat f_*.$$
Then $\phi_*([\gamma])=\phi_*\circ \pi_*([\widehat\gamma])=\pi_*\circ \widehat f_*([\widehat\gamma])=\pi_*([\widehat\eta])=[\eta],$ meaning that $\phi(\gamma)$ is 
homotopic to $\eta.$
 \end{proof}
It remains to prove Claim \ref{mer}\string:
\begin{proof}[Proof of Claim \ref{mer}] 
Notice that the image of the homology class of a longitude in $\partial \mathcal{N}_{\widehat\gamma}$ under $H_1(i_{\widehat\gamma})$ is not trivial in $H_1(M_{\widehat\gamma}; \mathbb{R})$ because $[{\widehat\gamma}]$ is not trivial in $H_1(PT^1\Sigma; \mathbb{R})$ by hypothesis. Then by considering the Mayer-Vietoris sequence in homology for the triad $(PT^1\Sigma, M_{\widehat\gamma}, \mathcal{N}_{\widehat\gamma}),$ we only need to show that there is an element in $H_2(PT^1\Sigma)$ whose image  under $\widehat\delta$, the connecting morphism of the sequence, is a non trivial element of $ H_1(\partial \mathcal{N}_{\widehat\gamma}).$
\vskip .2cm
Since the class of $\gamma$ is not trivial in $H_1(\Sigma; \mathbb{Z}),$  there is a non separating simple closed curve $\alpha$ such that $n[\alpha]=[\gamma]$ for some non zero integer $n$ (\cite{FM12}, Proposition 6.2). Choose $[\beta]$ a primitive element in $H_1(\Sigma; \mathbb{Z})$ such that its algebraic intersection with $[\alpha]$ is one. Consider $T_\beta$ the pre-image of $\beta,$ a simple closed geodesic, under the map $PT^1\Sigma\rightarrow\Sigma,$ which is a torus because $PT^1\Sigma$ is orientable.  Then $$\widehat\delta([T_\beta])=[\partial(T_\beta\cap \mathcal{N}_{\widehat\gamma})]=\langle[\beta],[\gamma]\rangle[m_{\widehat\gamma}]=n[m_{\widehat\gamma}],\hspace{.2cm}\mbox{with} \hspace{.2cm}n\neq 0.$$
\end{proof}

We conclude this subsection by giving a second proof of the fact that two closed geodesics in the same mapping class group orbit have homeomorphic canonical lift complements. This alternative proof relies on an action of the mapping class group of $\Sigma$ into its unit tangent bundle, which leaves invariant the foliation given by the geodesic flow.
\vskip .2cm
 Let $\phi$ be in $\Diff(\Sigma)$ and consider $\fun{\bar \phi}{\mathbb{H}^2}{\mathbb{H}^2}$ a lift  of $\phi.$ In that case we can give rise to a $\phi_*$-equivariant homeomorphism of the boundary at infinity of the hyperbolic space to itself,
 $$\fun{\partial \bar \phi}{\partial\mathbb{H}^2}{\partial\mathbb{H}^2}.$$
Using the fact that $T^1\mathbb{H}^2$ is diffeomorphic to
$$\{(a_1,a_2,a_3)\in\partial\mathbb{H}^2\times \partial\mathbb{H}^2\times \partial\mathbb{H}^2\mid a_i\neq a_j  \forall  i\neq j\} ,$$
quotiented via the fix-point free involution,
$$(a_1,a_2,a_3) \mapsto (a_2,a_1,a_3).$$
This diffeomorphism maps $(a_1,a_2,a_3)$ to the unique unit tangent vector $v$ normal to the geodesic in $\mathbb{H}^2$ with endpoints $a_1,a_2$  and pointing to $a_3.$
 \vskip .2cm
 From this point of view, we induce a $\phi_*$-equivariant map,
 $$\fun{\Phi}{T^1\mathbb{H}^2}{T^1\mathbb{H}^2}\hspace{.2cm}\mbox{such that} \hspace{.2cm} {\Phi}(x,y,z)=(\partial \bar \phi(x),\partial \bar \phi(y),\partial \bar \phi(z)) .$$
The $\phi_*$-equivariance property, implies that the map ${\Phi}$ descends to
 $$\fun{\widehat\Phi}{PT^1\Sigma}{PT^1\Sigma},$$
 sending the orbits of the geodesic flow to themselves.

\subsection{Hyperbolicity of  $M_{\widehat\gamma}$}

In this subsection we study some $3$-manifold properties of $M_{\widehat\gamma}.$ In particular we describe the JSJ-decomposition of $M_{\widehat\gamma},$ this is mainly due to Foulon and Hasselblatt (\cite{FH13}, Theorem 1.12), which states that the canonical lift complement over a filling closed geodesic is hyperbolic. This result is establish by verifying the hypotheses of the Hyperbolization Theorem \cite{Thu82}\string:

\begin{theorem*}[\bf{Hyperbolization Theorem}]
The interior of a closed $3$-mani\-fold $N$ with torus boundary admits a complete hyperbolic metric with finite volume if and only if it is irreducible, (homotopically) atoroidal, with infinite fundamental group and not homeomorphic to $\mathbb{S}^1\times \mathbb{D}^2,$ $\mathbb{T}^2\times[0,1]$ or $K\tilde\times[0,1]$ (the twisted interval bundle over the Klein bottle).
\end{theorem*}

We will begin by reviewing our list of definitions. For technical details see \cite{He76} and \cite{JS79}.
\vskip .2cm
Suppose $N$ is a compact, orientable $3$-manifold with boundary. An embedded surface in $N$ is said to be \textit{compressible} if there is an essential circle (that is, one that bounds no disk in the surface) on it that bounds a disk in $N.$ In the case where the embedded surface is not compressible, is said to be \textit{incompressible}. 

An incompressible torus is \textit{boundary-parallel} if there is an isotopy from it into $\partial N.$ That is, there is an embedded $\mathbb{T}^2 \times [0,1]$ such that $\mathbb{T}^2 \times \{0\}$ parametrizes the given torus and $\mathbb{T}^2 \times \{1\}$ a boundary component. 
 \vskip .2cm
The $3$-manifold $N$ is said to be \textit{irreducible} if every embedded $2$-sphere bounds a ball. In this case we also say that $N$ is \textit{homotopically atoroidal} if every $\pi_1$-injective map from the torus to $N$ is homotopic to a map into  $\partial N.$  Being homotopically atoroidal is a stronger property than being \textit{atoroidal }(here, atoroidal means that every incompressible torus in $N$ is boundary-parallel). Nevertheless, the two notions agree except for some Seifert-fibered manifolds (see \cite{JS79}, Lemma IV.2.6), which will not appear in this chapter. 
\vskip .2cm
Let us investigate for a given closed geodesic $\gamma$ on $\Sigma,$ which of properties on the Hyperbolization Theorem are satisfied by its canonical lift complement:
\begin{itemize}\setlength\itemsep{1em}
\item $M_{\widehat\gamma}$ is irreducible,  because if $\mathbb{S}^2$ is embedded in $M_{\widehat\gamma}$ and hence into $PT^1\Sigma$, then it bounds a ball $B$ in $PT^1\Sigma,$ by irreducibility of $PT^1\Sigma.$  As $\mathbb{S}^2\cap\widehat\gamma=\emptyset$  we conclude that $\widehat\gamma$ is embedded outside $B,$ because if $\widehat\gamma$ is inside $B,$ then its projection to $\Sigma$  would be nullhomotopic. 
\item $\pi_1(M_{\widehat\gamma})$ is infinite, because the inclusion map $\fun{i}{M_{\widehat\gamma}}{PT^1\Sigma}$ induces,  by transversality, a $\pi_1$-surjective map and $\pi_1(PT^1\Sigma)$ is infinite.
\end{itemize}
\begin{definition}\label{split} \normalfont 
Let $\eta$ be a simple closed geodesic on $\Sigma,$ then we define the following embedded surfaces on $M_{\widehat\gamma}$\string:
$$T_\eta  \hspace{.2cm}  \mbox{the pre-image of} \hspace{.2cm}  \eta \hspace{.2cm} \mbox{under the map} \hspace{.2cm} PT^1\Sigma\rightarrow\Sigma,\hspace{.2cm}\mbox{if} \hspace{.2cm} i(\eta,\gamma)=0,$$
and
$$(T_\eta)_{\widehat\gamma}:=(T_\eta)\setminus \mathcal{N}_{\widehat\gamma}, \hspace{.2cm} \mbox{if} \hspace{.2cm}i(\eta,\gamma)\neq 0.$$
\end{definition}
\begin{itemize}\setlength\itemsep{1em}
\item For the atoroidality of $M_{\widehat\gamma}$, notice that if a simple closed geodesic $\eta$ does not intersects $\gamma$ in $\Sigma.$ Then $T_\eta$ is an embedded incompressible torus in $M_{\widehat\gamma},$ (\cite{FH13}, Theorem B.8) which is not boundary-parallel. Therefore, a necessary condition for $M_{\widehat\gamma}$ to be atoroidal is that $\gamma$ fills $\Sigma.$ 
\item Conversely, Foulon and  Hasselblatt proved in (\cite{FH13}, Theorem 1.12) with the advising of Calegari, Fenley, Otal and Walsh, that the filling condition on $\gamma$ is sufficient to provide atoroidality to $M_{\widehat\gamma},$ and consequently a unique hyperbolic structure with finite volume.
\end{itemize}

\begin{theorem*}[Hyperbolic Structures on Canonical Lift Complements]
If $\gamma$ is a filling self-transverse closed curve with minimal self-intersection number on a hyperbolic surface $\Sigma,$ and $\widehat\gamma$ the canonical lift of $\gamma.$ Then $M_{\widehat\gamma}$ is a hyperbolic manifold of finite volume.
\end{theorem*}

In the case where $\gamma$ is just a non simple closed geodesic, that is $M_{\widehat\gamma}$ is an irreducible $3$-manifold with torus boundary, then by JSJ-decomposition Theorem \cite{JS178,Joh79}, we can canonically split $M_{\widehat\gamma}$ along a finite (maybe empty) collection of disjoint and non-parallel nor boundary-parallel incompressible embedded tori into connected components which are either atoroidal or Seifert fibered\string:

\begin{theorem*}[\bf{JSJ-decomposition}]
Let $N$ be a compact, irreducible $3$-mani\-fold with boundary. In the interior of $N,$ there exists a family $C = \{T_1,...,T_r\}$ of disjoint tori that are incompressible and not boundary parallel, with the following properties\string:
\begin{enumerate}
\item each connected component of $N\setminus C$ is either a Seifert manifold or is atoroidal;
\item the family $C$ is minimal among those satisfying $(1).$
\end{enumerate}
Moreover, such a family $C$ is unique up to ambient isotopy.
\end{theorem*}
We recall that the \textit{characteristic submanifold of} $N$ is the union of all Seifert fibered JSJ-components of $N.$

\begin{corollary}\label{JSJGAMMA}
Let $\gamma$ be a non simple closed geodesic in $\Sigma.$  Denote by $\Sigma_1$ be the essential subsurface filled by $\gamma,$  $\Sigma_2$ the closure in $\Sigma$ of $\Sigma \setminus  \Sigma_1$ and $M_{1\widehat\gamma}:=PT^1\Sigma_1\setminus  \mathcal{N}_{\widehat\gamma}.$ Then the JSJ-decomposition of $M_{\widehat\gamma}$ in its atoroidal pieces and characteristic submanifold is\string: 
$$M_{1\widehat\gamma}\cup PT^1\Sigma_2,\hspace{.2cm} \mbox{with the respective splitting incompresible tori} \hspace{.2cm}T_{(\partial \Sigma_1\setminus \partial\Sigma)}.$$   \hfill $\square$
\end{corollary}

\begin{figure}[h]
\centering
\includegraphics[scale=0.4] {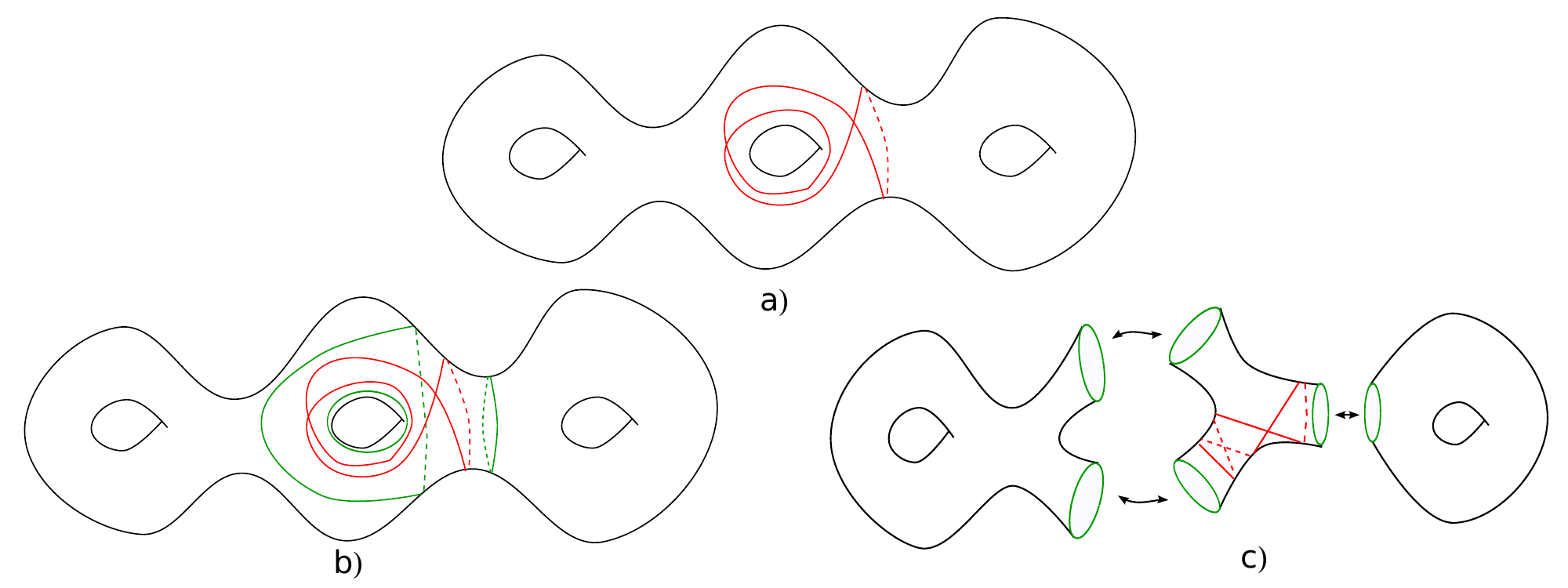}
\caption{ $a)$A non filling closed goedesic $\gamma$ on $\Sigma.$ \hspace{.2cm} $b)$ Finding the essential subsurface on $\Sigma$ filled by $\gamma.$ \hspace{.2cm} $c)$The projection of the JSJ-decomposition  of $M_{\widehat\gamma}.$
}\label{jsj}
\end{figure}

Lastly, we point out that the embedded surfaces $(T_\eta)_{\widehat\gamma}$ on Definition \ref{split} are also incompressible on $M_{\widehat\gamma}.$ These surfaces, will be important for our main result in section \ref{s3}.

\begin{lemma}\label{irr}
For every pair of intersecting closed geodesic $\eta$ and $\gamma$ in $\Sigma,$ where $\eta$ is simple, the embedded surface $(T_\eta)_{\widehat\gamma}$ in $M_{\widehat\gamma}$ is $ \pi_1$-injective.
\end{lemma}

\begin{proof}[\bf{Proof}]
 We consider, for the sake of concreteness, the case where the surface $(T_\eta)_{\widehat\gamma}$ is separating and leave the non separating case to the reader. Let us split  $M_{\widehat\gamma}$ by $(T_\eta)_{\widehat\gamma},$ into pieces $N_1$ and $N_2.$ By Van Kampen it  is enough to show the $\pi_1$-injectivity for the surface $(T_\eta)_{\widehat\gamma}$ in $N_i.$
\vskip .2cm
If the surface  $(T_\eta)_{\widehat\gamma}$ is not $\pi_1$-injective in $N_i,$ then by the Loop Theorem, there is a disk $D_0$ whose interior is in the interior of $N_i$ and whose innermost intersection with $(T_\eta)_{\widehat\gamma}$ is an essential simple closed curve $\alpha$ in $(T_\eta)_{\widehat\gamma}.$ 
\vskip .2cm
As $T_\eta$ is an incompressible surface in $PT^1\Sigma$ (\cite{FH13}, Theorem B.8),  then  $\alpha$ bounds a disk $D_1$ in $T_\eta$ that intersects ${\widehat\gamma}.$  So by the irreducibility of $PT^1\Sigma,$ the embedded sphere formed by $D_0\cup D_1$  would bound a ball $B$ in $PT^1\Sigma.$
\vskip .2cm
By the periodicity of $\widehat\gamma,$ there is at least one arc of $\widehat\gamma$ inside $B$ with endpoints at $D_1.$  This arc is homotopic relative to the boundary to an arc in $B$ that union  with a simple arc in $D_1$ with the same endpoints, bounds a disk in $N_i.$ Such homotopy inside $PT^1\Sigma,$ induces a homotopy on $\Sigma$ that reduces the intersection number between $\gamma$ and $\eta,$ contradicting the fact that $\gamma$ and $\eta$ are in minimal position. \end{proof}

\section[Sequences of geodesics with bounded volume complement] {Sequences of closed geodesics with uniformly bounded volume complement}\label{s2}

In this section we prove one of our main results, Theorem \ref{2}, which provides a way of constructing sequences of closed geodesics of increasing length, whose canonical lift complements are not homeomorphic and whose volumes are universally bounded by the volume of a link complement on $PT^1\Sigma.$ Although this kind of examples already existed in the literature (see \cite{BPS16}, Subsection 3.2  or \cite{BPS17}, Example 5.2), we point out that our method generalizes the previous examples.

\begin{reptheorem}{2}
Given a hyperbolic surface $\Sigma,$ there exist a constant $V_0>0$ and a sequence $\{\gamma_n\}$ of filling closed geodesic on $\Sigma,$ with $M_{\widehat\gamma_n}\not\cong M_{\widehat\gamma_k}$ for every $k\not= n,$ such that $\Vol(M_{\widehat\gamma_n})< V_0$ for every $n\in\mathbb{N}.$  Moreover, for any sequence $\{X_n\}$ of hyperbolic metrics on $\Sigma,$  we have that $\ell_{X_n}(\gamma_n) \nearrow \infty.$ \end{reptheorem}

\begin{proof}[\bf{Proof}]
Choose $\gamma_0$ and $\eta$ closed geodesics on $\Sigma$ such that the first one is filling, the second has a fixed orientation and $i(\gamma_0,\eta)>1.$
\vskip .2cm
For each intersection point  between $\gamma_0$ and $\eta,$ we measure the angle of intersection at that point by using the orientation of $\eta$ and the right side on $\Sigma$ with respect to $\eta.$ Let $p$ be a point of intersection of $\gamma_0$ with $\eta$ that minimizes the angle. 
\vskip .2cm
Define $\gamma_n$ the unique closed geodesic homotopic to the closed curve that starts at $p,$ travels $n$ times along $\eta$ with the same orientation and then one time about $\gamma_0,$ that is
$$\eta^n\ast_p \gamma_0  \hspace{.2cm} \mbox{(denoted as} \hspace{.2cm}\gamma_{n,0} ).$$
Notice that $\gamma_{n,0}$ is not self-transverse, so we will slightly modify it to obtain one that is, in the following manner:
\vskip .2cm
Choose a $\delta$ neighborhood of ${\eta}$  such that all self-intersections of $\gamma_0$ that lie outside the geodesic ${\eta},$ also lie outside the $\delta$ neighborhood of ${\eta}.$ 
\vskip .2cm
 Erase the geodesic arc of $\gamma_{0}$  to which $p$ belongs and is contained in the $\delta$ neighborhood of ${\eta}.$ Link the extremal points by the geodesic segment ${\eta_n}$ in the $\delta$ neighborhood of ${\eta}$ that winds $n$ times in the orientation of ${\eta}.$ Let us denote this piecewise geodesic  as $\gamma_{n,1}.$
\vskip .2cm
\begin{figure}[h]
\centering
\includegraphics[scale=0.4] {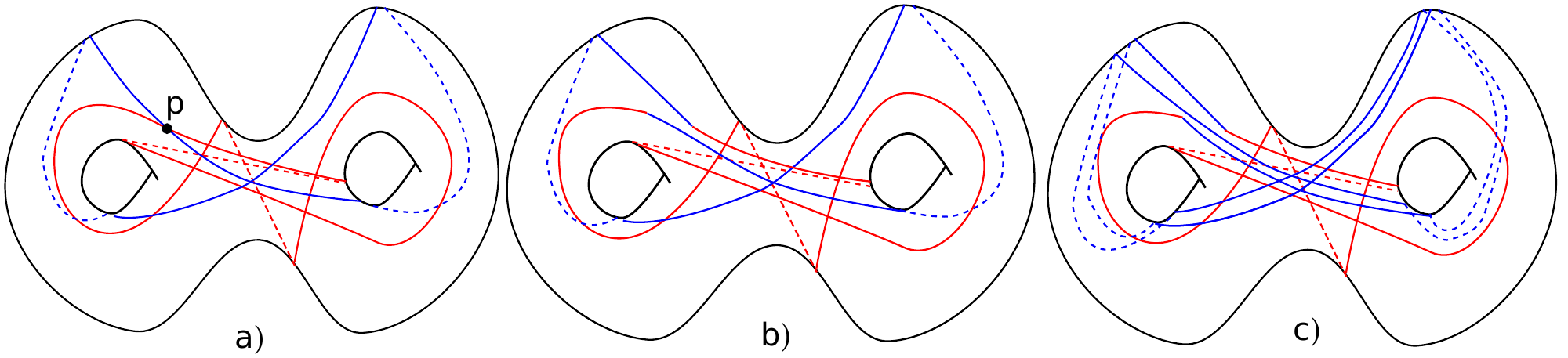}
\caption{$a)$ $\gamma_{0}\cup\eta,$ \hspace{.2cm} $b)$ $\gamma_{1,1},$  \hspace{.2cm} and \hspace{.2cm} $c)$ $\gamma_{2,1}$ on $\Sigma.$
}\label{fill+}
\end{figure}
Using elementary hyperbolic geometric, one easily shows that $\gamma_{n,1}$ has no self-tangency points, and the angle in the intersection points between $\gamma_{n,1}\setminus\eta_n$  and $\eta_n$ are less than the angles in the gluing points.
\vskip .2cm
Notice that by construction, $$n^2i(\eta,\eta)+n(i(\gamma_0,\eta)-1)= \sharp (\gamma_{n,1}\cap\gamma_{n,1}).$$

\begin{claim}\label{si}
$\gamma_{n,1}$ has the minimal number of self-intersection.
\end{claim}

\begin{figure}[h]
\centering
\includegraphics[scale=0.6] {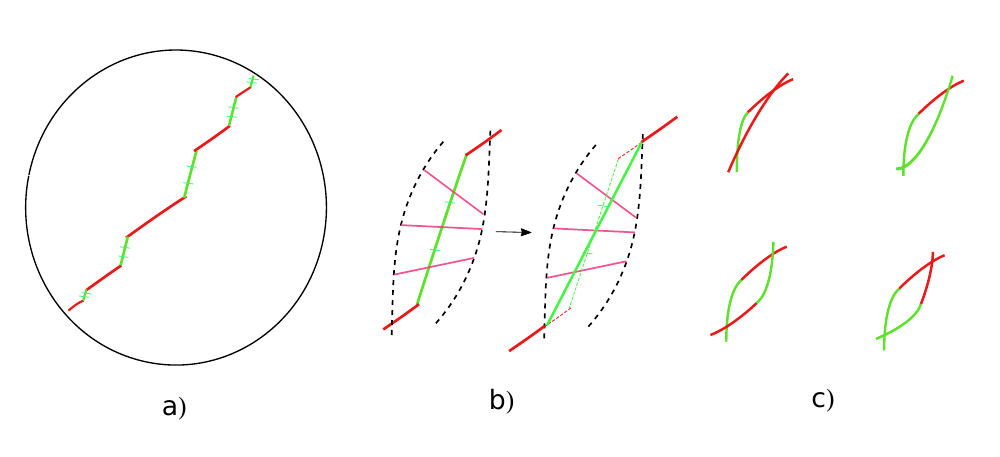}
\caption{a)A lift of $\gamma_{n,0}$ in $\mathbb{H}^2.$ \hspace{.05cm} b)The modification from $\widetilde{\gamma_{n,0}}$ to $\widetilde{\gamma_{n,1}}.$  \hspace{.05cm} c)Possible bigons given by the interesction of two lifts of  $\gamma_{n,1}.$ 
}\label{pd}
\end{figure}
\begin{proof}[Proof]By Theorem \ref{2-gone}, the proof reduces to show that $\gamma_{n,1}$ has no singular $1$-gons nor $2$-gons. Consider the lifts of $\gamma_{n,1}$ in the universal cover of $\Sigma,$  which are piecewise geodesic lines in $\mathbb{H}^2.$ It is enough to show that every pair of them intersects at most once. Indeed, if there were two consecutive intersections they cannot happen in segments of the same geodesic segment lifts, so in between there must is at least one gluing point of two distinct geodesic segments.
\vskip .2cm
 If there is more than one gluing point in between the consecutive intersections, we would have geodesic side quadrilaterals whose sum of interior angles is bigger or equal to $2\pi,$ which is a contradiction in the hyperbolic plane.
\vskip .2cm
The case were there is exactly one gluing point cannot happen thanks to the choice of $p$ as the intersection points between $\gamma_0$ and $\eta$ with minimal angle.
\end{proof}
In order to make $\gamma_{n,1}$ smooth, just round the two corners of $\gamma_{n,1},$ and denote it by the same name. This gives us a self-transverse closed curve homotopically transversal to $\gamma_n,$ which means that $M_{\widehat\gamma_n}\cong M_{\widehat{\gamma_{n,1}}}.$ By Theorem \ref{charac}, $M_{\widehat\gamma_n}\not\cong M_{\widehat\gamma_k}$ for every $k\not= n,$ because their self-intersection number, which is an invariant under the mapping class group action, is different for each element of the sequence $\{\gamma_n\}.$ Moreover, for any sequence $\{X_n\}$ of hyperbolic metrics on $\Sigma,$  we have that $\ell_{X_n}(\gamma_n) \nearrow \infty,$ because the self-intersection number is increasing too. 
\vskip .2cm
The hyperbolicity of $M_{\widehat\gamma_n}$ comes from the following fact:

\begin{claim}\label{fill}
$\gamma_n$ is filling for every $n\in \mathbb{N}.$\end{claim}
\begin{figure}[h]
\centering
\includegraphics[scale=0.5] {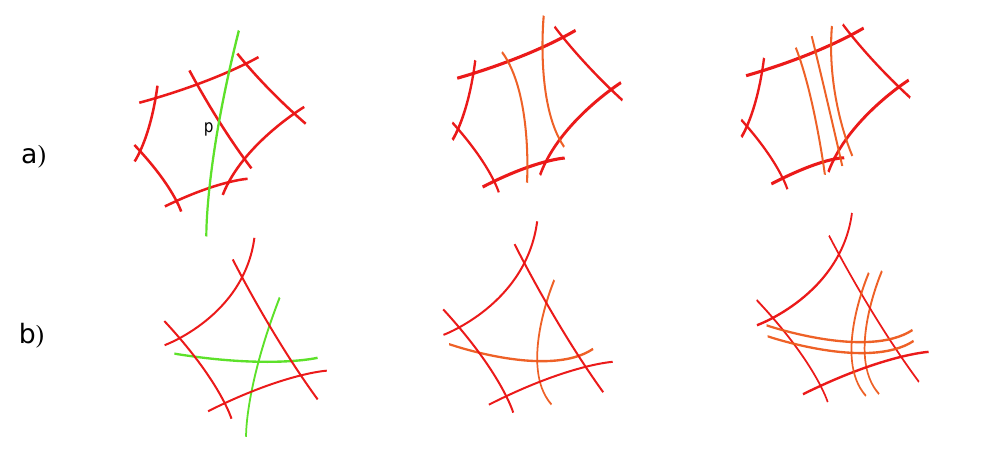}
\caption{The images on the left are two examples of partitions of $\Sigma$ via $\eta\cup\gamma_0$, on the center $\gamma_1$, and on right for $\gamma_2.$ \hspace{.2cm} a) The partition around $p.$ \hspace{.2cm} b)The partition where the arcs does not contain $p.$ 
}\label{pd}
\end{figure}
\begin{proof}[Proof] 
Take the partition of $\Sigma$ by  $\gamma_0$ and draw $\eta$ over this partition. First we will be interested in comparing the partitions between $\Sigma\setminus\gamma_0$ and $\Sigma\setminus\gamma_{1,1}.$
\vskip .2cm
Consider the four sides of the graph induced by $\gamma_0\cup\eta$ that have $p$ as vertex and cut them at $p$ and glue them in a particular order that depends on the orientation of $\eta.$ Next, superpose the pair of new arcs to geodesic segments on $\eta_1.$
\vskip .2cm
In the case were the boundary discs of the partition of  $\Sigma$ by $\gamma_0$  does not contain $p$ we simply consider the partition of the discs given by $\gamma_0\cup\eta_1.$
\vskip .2cm
Thus if $\gamma_0$ is filling then $\gamma_{1,1}$ is also filling. Moreover, as this property of being filling is preserved by transversal homotopy (Corollary \ref{comp}), implies that $\gamma_1$ is also filling.
\vskip .2cm
Notice that, $\Sigma\setminus\gamma_{n,1}$ is a subdivision of the partition given by $\gamma_{1,1}$, then $\gamma_n$ is also filling.
\end{proof}

Finally, in order to find an upper bound to the volume of all the hyperbolic $3$-manifolds $M_{\widehat\gamma_{n}}$ we will construct $L_0$ and $L_1,$ two knots inside $M_{\widehat\gamma_{0}}$ and show that each $M_{\widehat\gamma_{n}}$ is obtained by making Dehn surgery along $L_0$ and $L_1.$ As Dehn filling does not increase the simplicial volume (\cite{Thu79} Theorem 6.5.4), then\string:
$$\Vol(M_{\widehat\gamma_{n}})\leq \| M_{\widehat\gamma_{{0}}}\setminus\{L_0,L_1\}\| ,\hspace{.2cm}\mbox{for every}  \hspace{.2cm}k\in \mathbb{N}.$$

For the construction of $L_0$ and $L_1,$ let $\theta$ be the angle of intersection of $\gamma_{0}$ and $\eta$ at $p.$  As the number of possible angles in the intersection of $\gamma_{0}$ and $\eta$ is finite, there exist a small $\varepsilon$ such that the following annulus:
$$A:=\{ (\eta (t),\dot\eta(t) e^{2i\pi s}) \hspace{.2cm} | \hspace{.2cm} t\in\mathbb{S}^1,\hspace{.2cm} s\in[0,\theta+ \varepsilon] \}\subset PT^1\Sigma,$$
is intersected only once by $\widehat\gamma_{0}$ at the point whose projection on $\Sigma$  is $p.$ Denote the boundary component of  $A$ associated to $\widehat\eta$ as $L_0$ and let $L_1$ be the other boundary component of $A.$ 
\vskip .2cm
Take $V$ a solid torus inside the preimage the $\delta $ neighborhood of $\eta$ under the map $\Sigma\rightarrow PT^1\Sigma$ which contains $A$ and does not intersect other arcs of $\widehat\gamma_{0}$ different from the one that intersects $A.$ 
\vskip .2cm
Let $V'$  be the resulting $3$-manifold obtained by performing a $\frac{1}{n}$ Dehn surgery on one component of $\partial A$ and  a $\frac{-1}{n}$  Dehn surgery on the other component of $\partial A$ inside $V.$ By (\cite{Os06}, Theorem 2.1) there exist a diffeomorphism $\phi_n:V'\rightarrow V$ such that $\phi_n\mid_{\partial V'}=id.$ Moreover, each homotopy class relative to the boundary of an arc transversal to the annulus is send to the homotopy class of the initial arc concatenated, at the point of intersection with the annulus, with an arc that winds along the core of $V$ $n$ times.
\vskip .2cm
In conclusion, the diffeomorphism $\phi_n$ extends to a self-homeomorphism of $PT^1\Sigma\setminus\{L_0, L_1\}$ that sends the isotopy class of $\widehat\gamma_{0}$ to $\widehat\gamma_{n}.$
\end{proof}
\begin{note*}From Theorem \ref{2}, it is natural to look for bounds on the volume of $M_{\widehat{\gamma}}$ invariant under the mapping class group action and more intrinsic than the self-intersection number. 
\end{note*}
We can also apply our method in Theorem \ref{2} to explain why the volume of the sequence of geodesics found in (Example 5.2, \cite{BPS17}) is bounded. We recall that every closed geodesics on the modular surface is represented as a positive word in $x=\begin{pmatrix} 
1 & 1 \\
0 & 1 
\end{pmatrix}$ 
and $y=\begin{pmatrix} 
1 & 0 \\
1 & 1 
\end{pmatrix},$ containing both symbols. Conversely, any such word encodes a unique periodic geodesic \cite{Ser85}. In that case, we refer to the code for a closed geodesic $\gamma$ as $\omega(\gamma)$ and denote $n_{\gamma}$ the number of (cyclic) subwords of the form $xy$ in $\omega(\gamma).$ This coding arises from a continued fraction expansion (see \cite{Ser85}) where $n_{\gamma}$ is exactly half the period of the (even) continued fraction corresponding to $\gamma.$

\begin{repcorollary}{mod1}
For the modular surface $\Sigma_{mod},$ there exist a constant $V_0>0$ and a sequence $\{\gamma_k\}$ of filling closed geodesics on $\Sigma_{mod},$  with $M_{\widehat\gamma_n}\not\cong M_{\widehat\gamma_k}$ for every $k\not= n,$ such that $\Vol(M_{\widehat\gamma_k})< V_0$ for every $k\in\mathbb{N}$ and the period of the continued fraction expansion of $\gamma_k$ tends to infinity.
\end{repcorollary}

\begin{proof}[\bf{Proof}]
Consider the closed geodesics $\gamma_k$ whose code is of the form $\{x^2y\alpha^k\},$ where $\alpha$ is any positive word containing $x$ and $y$ that ends with $y,$ whose representing geodesic intersects the geodesic encoded by $x^2y$ more than once.  Proceeding as in Theorem \ref{2} one shows that the volume associated to $x^2y\alpha^k$ is bounded by the volume of $$PT^1\Sigma_{mod}\setminus (\widehat{x^2y} \cup \widehat{\alpha} \cup \widehat{\alpha_{\theta}}),$$
where $\widehat{\alpha_{\theta}}$ is a knot obtained by translating $\widehat\alpha$ by some small angle $\theta$ in the fiber direction of $PT^1\Sigma_{mod}.$
\end{proof}

\begin{note}\label{constant-homotopyarcs}
For the sequence of geodesics in Theorem \ref{2}, we observe that if we fix a suitable pants decomposition on $\Sigma,$ then the homotopy classes of arcs in each pair of pants is bounded.  The only thing that changes is the number of arcs in the fixed homotopy classes.
\end{note}

\begin{figure}[h]
\centering
\includegraphics[scale=0.4] {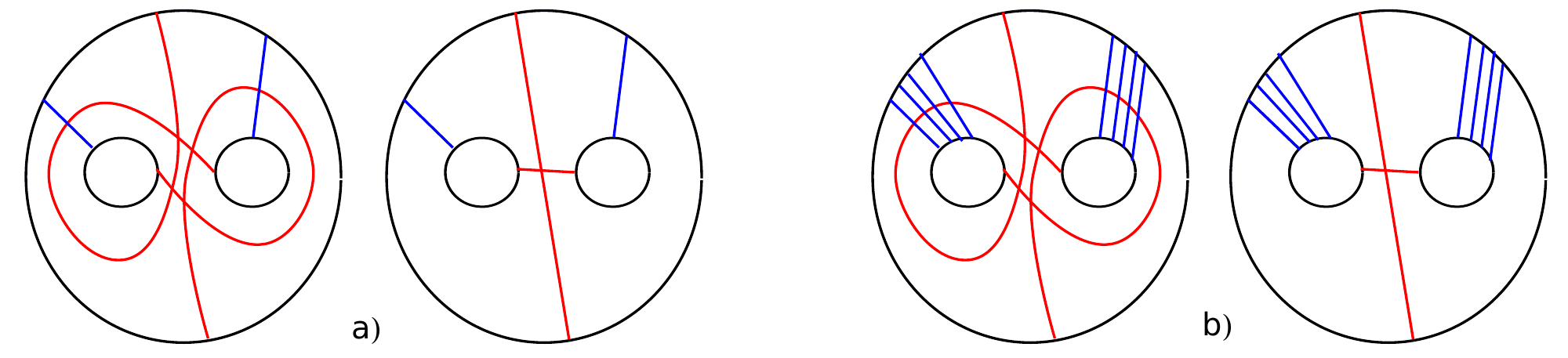}
\caption{$a)$ After cutting $\Sigma$ (see Figure \ref{fill+}b) along a pants decomposition and taking a minimal position homotopy class of the remaining  $\gamma_{1,1}$-arcs relative to the pants boundary, and \hspace{.2cm} $b)$ the corresponding operation for $\gamma_{4,1}$ in $\Sigma.$ 
}\label{piste}
\end{figure}

This remark is the motivation for our next main result, that we present in the following section.

\section{Lower bound on the volume of $M_{\widehat\gamma}$}\label{s3}

In this section we prove a lower bound for the volume of the canonical lift complement  (Theorem \ref{1}). The lower bound is obtained in terms of combinatorial data coming from the geodesic and a pants decomposition of the surface. We apply this result to find a sequence of geodesics in the modular surface where the lower bound can be written in terms of the period of the continued fraction extension of the geodesics (Corollary \ref{mod}). We conclude by finding a sequence of geodesics on surfaces, whose lower bound is written in terms of their length (Theorem \ref{pi-b-2}).

\begin{reptheorem}{1}
Given a pants decomposition $\Pi$ on a hyperbolic surface $\Sigma,$  and $\gamma$ a filling geodesic, we have that\string:
 $$\Vol(M_{\widehat\gamma})\geq \frac{v_3}{2}\sum_{P \in \Pi}(\sharp\{\mbox{homotopy classes of} \hspace{.2cm}  \gamma\mbox{-arcs in} \hspace{.2cm} P\}-3),$$ 
  where $v_3$ is the volume of a regular ideal tetrahedra.
\end{reptheorem}

Given a pair of pants $P,$ we say that two arcs $\alpha,\beta:[0,1] \rightarrow P$ with $\alpha(\{0,1\})\cup \beta(\{0,1\})\subset \partial P$ are in the same homotopy class in $P,$ if there exist an homotopy $\fun{h}{[0,1]_1\times[0,1]_2}{P}$ such that\string: $$h_0(t_2)=\alpha(t_2),  \hspace{.2cm} h_1(t_2)=\beta(t_2) \hspace{.2cm} \mbox{and} \hspace{.2cm}h([0,1]_1\times\{0,1\})\subset \partial P.$$

 \begin{note}\label{bda} 
Up to isotopy, for a family of simple arcs without intersection there are only six configurations of arcs in $P$. These are shown in Figure \ref{badarcs}. The $3$ in the lower bound of Theorem \ref{1} comes from the fact that such configurations have at most $3$ homotopy classes of $\gamma$-arcs on $P.$
 \begin{figure}[h]
\centering
\includegraphics[scale=0.3] {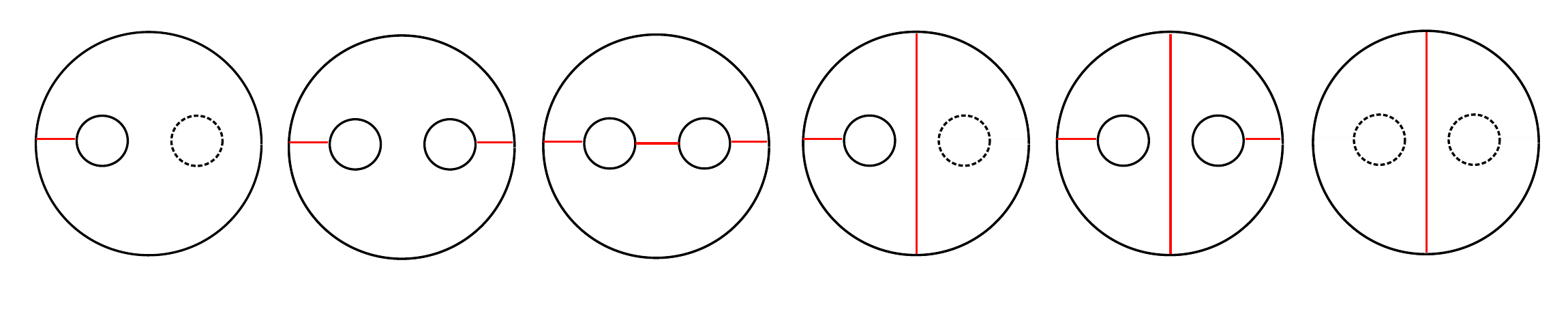}
\caption{ The only six $\gamma$-arcs configuration on $P$ up to homotopy classes whose $\gamma$-arcs are simple arcs without intersections.
}\label{badarcs}
\end{figure} 
\end{note}

Before stating the main result to prove Theorem \ref{1}  we recall some definitions.
\vskip .2cm
If $N$ is a hyperbolic $3$-manifold and $S \subset N$ is an embedded incompressible surface, we will use $N\backslash\backslash S$ to denote the manifold that is obtained by cutting along $S;$ it is homeomorphic to the complement in $N$ of an open regular neighborhood of $S.$ If one takes two copies of $N\backslash\backslash S,$ and glues them along their boundary by using the identity diffeomorphism, one obtains the double of $N\backslash\backslash S,$ which is denoted by $D(N\backslash\backslash S).$

\begin{definition}\label{DP}  \normalfont 
Let $P$ be a pair of pants which belongs to a pants decomposition on a surface $\Sigma$ and $\gamma$ a closed geodesic having points in $\Sigma\setminus P$ such that $P\cap\gamma$ is  finite set of geodesic arcs $\{\alpha_i\}$ connecting some boundary components of $P.$ Then denote,
$$P_{\widehat\gamma} \hspace{.2cm} \mbox{as} \hspace{.2cm} PT^1P\setminus\bigcup_i \widehat\alpha_i\cong (\mathbb{S}^1\times P)\setminus\bigcup_i \widehat\alpha_i.$$
And define,
$$D({P}_{\widehat\gamma}),$$ as gluing two copies of $PT^1P\setminus\bigcup_i \widehat\alpha_i$ along the punctured tori coming from $$\partial PT^1P\setminus\big(\partial PT^1P\cap\big(\bigcup_i  \widehat\alpha_i\big)\big), $$
by using the identity. Moreover, $D({P}_{\widehat\gamma})$ is homeomorphic to
$$(\mathbb{S}^1\times  S^0)\setminus \bigcup_i  D(\widehat\alpha_i),$$
where $S^0$ is a surface of genus two (if $\sharp(\partial\Sigma\cap \partial P)=0)$, or a surface of type $(1,2)$ (if $\sharp(\partial\Sigma\cap \partial P)=2)$ or a surface of type $(0,4)$ (if $\sharp(\partial\Sigma\cap \partial P)=1),$ and $D(\widehat\alpha_i)$ is an embedded closed curve on $\mathbb{S}^1\times  S^0$ obtained by gluing $\widehat\alpha_i$ along the two points $\partial PT^1P\cap\widehat\alpha_i$ by the identity.
\end{definition}

\begin{figure}[h]
\centering
\includegraphics[scale=0.4] {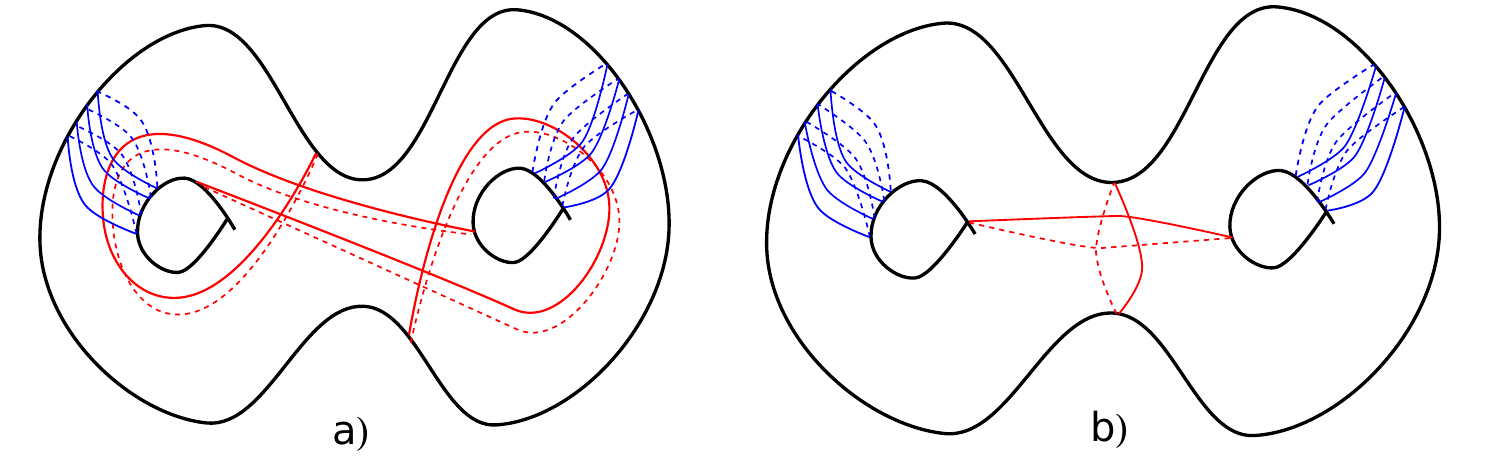}
\caption{The projection of  $D({P}_{\widehat{\gamma_{4,1}}})$ (after an homotopy of $\gamma_{4,1}$-arcs to a minimal position configuration) over $ S^0$ for Figure \ref{piste}b where: $a)$  is the right pair of pants and \hspace{.2cm} $b)$ is the left pair of pants.
}\label{DPg}
\end{figure}

Let $M$ be a connected, orientable 3-manifold with boundary and let $S(M;\mathbb{R})$ be the \textit{singular chain complex} of $M.$ More concretely, $S_k(M;\mathbb{R})$ is the set of formal linear combination of $k$-simplices, and we set as usual $S_k(M, \partial M;\mathbb{R})= S_k(M;\mathbb{R})/S_k(\partial M;\mathbb{R})$. We denote by 
$\|c \|$ the $l_1$-norm of  the $k$-chain $c$.
If $\alpha$ is a homology class in $H^{sing}_k(M, \partial M;\mathbb{R})$, the \textit{Gromov norm of $\alpha$}  is defined as:
$$\|\alpha\|=\inf_{[c]=\alpha} \{ \|c \|=\sum_{\sigma}|r_\sigma|\hspace{.1cm}\mbox{such that}\hspace{.1cm} c=\sum_{\sigma}r_\sigma \sigma\}.$$
The \textit{simplicial volume} of $M$ is the Gromov norm of the fundamental class of $(M,\partial M)$ in  $H^{sing}_3(M,\partial M;\mathbb{R})$ and is denoted by $\|M\|.$
\vskip .2cm
The key ingredient to prove Theorem \ref{1} is the following result due to Agol, Storm and Thurston (\cite{AST07}, Theorem 9.1)\string:

 \begin{theorem*}[Agol-Storm-Thurston]
 Let $N$ be a compact manifold with interior a hyperbolic $3$-manifold of finite volume. Let $S$ be an embedded incompressible surface in $N.$ Then
 $$\Vol(N)\geq  \frac{v_3}{2} \| D(N\backslash\backslash S)\|.$$
  \end{theorem*}
 
 We now prove the lower bound for the volume of the canonical lift complement\string:
 \begin{proof}[\bf{Proof of Theorem \ref{1}}]
Let $\Sigma$ be a hyperbolic surface of genus $g$ and $n$ boundary components and let $\{\eta_i\}^{3g+n-3}_{i=1}$ be the simple closed geodesics inducing $\Pi$ a pants decomposition on $\Sigma.$ Consider the  incompressible surface  $S:= \bigsqcup^{3g+n-3}_{i=1} {(T_{\eta_i})_{\widehat\gamma}}$ in $M_{\widehat\gamma}$ (Lemma \ref{irr}). From (\cite{AST07}, Theorem 9.1) we deduce that\string:
$$\Vol(M_{\widehat\gamma}) \geq\frac{v_3}{2} \| D(M_{\widehat\gamma}\backslash\backslash S)\|= \frac{v_3}{2}\sum_{P \in \Pi} \| D({P}_{\widehat\gamma})\|.$$
For each pair of pants $P$ we have\string:
$$v_3\sharp\{\mbox{cusps of}  \hspace{.2cm} D({P}_{\widehat\gamma})^{hyp}\} \leq  \Vol( D({P}_{\widehat\gamma})^{hyp})\leq v_3\| D({P}_{\widehat\gamma})^{hyp}\| = v_3\| D({P}_{\bar\gamma})\|$$
 where $D({P}_{\bar\gamma})^{hyp}$ is the atoroidal piece of $D({P}_{\widehat\gamma}),$   i.e., the complement of the characteristic sub-manifold, with respect to its JSJ-decomposition. The first and second inequality come from \cite{Ada88}  and \cite{Gro82} respectively. 
 \vskip .2cm
 Notice that if $\omega_1$ and $\omega_2$ are a pair of homotopic $\gamma$-arcs on $P$ then their respective canonical lifts  $\widehat\omega_1$ and $\widehat\omega_2$ are isotopic $\widehat\gamma$-arcs in $PT^1P.$ Indeed, let $\widetilde{\omega_1}$ and  $\widetilde{\omega_2}$ be lifts on the universal cover starting in the same fundamental domain. Take an homotopy of geodesics that varies from  $\widetilde{\omega_1}$ to $\widetilde{\omega_2}$ and project this homotopy to $P.$ This will give us a homotopy $h$ of geodesics arcs $h_t$ that start in $\omega_1$ and end in $\omega_2.$ The image of the geodesic homotopy does not intersects other $\widehat{\gamma}$-arcs because of local uniqueness of the geodesics. Then we have that the geodesic homotopy induces an isotopy in $PT^1P$ between their corresponding canonical lifts. Moreover, the image of the geodesic homotopy induces a co-bounding annulus between the corresponding knots coming from the double of the corresponding canonical lifts in $D(PT^1P).$  Therefore contributing to a Seifert-fibered component, where the JSJ-decomposition separates this set of parallel knots from the rest of the manifold (see Figure \ref{DPgJSJ}).
 \vskip .2cm
 Let $\Omega$ be the subset of $\gamma$-arcs on $P$ having one arc for each homotopy class of $\gamma$-arcs on $P$. This means that $D({P}_{\widehat\gamma})^{hyp}\cong D({P}_{\widehat{\Omega}})^{hyp}.$ Moreover, $D({P}_{\widehat{ \Omega}})$ can be seen as a link complement in $\mathbb{S}^1\times  S^0$, see Definition \ref{DP}, whose projection to $S^0$ is a union of closed loops transversally homotopic to a union closed loops in minimal position. By using a multicurve version of a result due to Kra (\cite{IM02}, Theorem 1) which is an analog result for the hyperbolicity and JSJ-decomposition of $M_{\widehat\gamma},$ we have that the atoroidal piece of $D({P}_{\widehat{\Omega}})$ corresponds to the subsurface of $S^0$ which $D(\Omega)$ fills.

  \begin{enumerate}
   \item If the $\Omega$-arc configuration on $P$ is in the list of Remark \ref{bda}, then we have that $D({P}_{\widehat\gamma})^{hyp}=\emptyset$  and Remark \ref{bda} also gives us:
 $$v_3(\sharp\{\mbox{homotopy classes of} \hspace{.2cm}  \gamma\mbox{-arcs in} \hspace{.2cm} P\}-3)\leq v_3\sharp\{\mbox{cusps of}  \hspace{.2cm} D({P}_{\widehat\gamma})^{hyp}\}.$$
 \item If the $\Omega$-arc configuration on $P$ is not in the list of Remark \ref{bda}, then there is at least one geometric intersection point on the projection of the link complement $D({P}_{\widehat{ \Omega}})$  to $S^0.$  \end{enumerate} 
\vskip .2cm

\begin{figure}[h]
\centering
\includegraphics[scale=0.4] {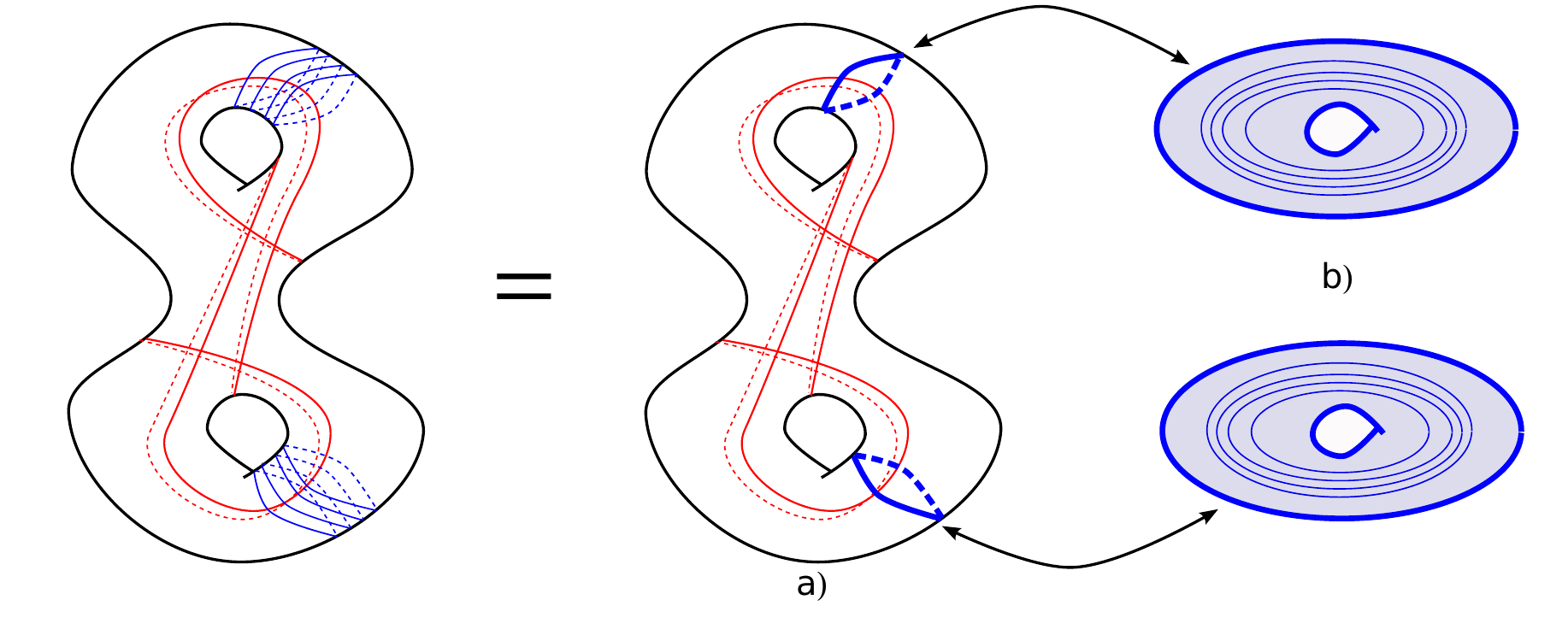}
\caption{The JSJ-decomposition of $D({P}_{\widehat{\gamma}})$ of Figure \ref{piste}c.
}\label{DPgJSJ}
\end{figure}
By (\cite{IM02}, Theorem 1) we conclude that $D({P}_{\widehat\gamma})^{hyp}\neq\emptyset.$ We will now define an injective function: 
$$\left\{ \begin{array}{l} \widehat\gamma\mbox{-arcs in} \\ \hspace{.5cm} PT^1P\end{array}\right\} \overset{\varphi}\longrightarrow \left\{ \begin{array}{l}\hspace{.1cm} \mbox{cusps of} \\ D({P}_{\widehat\gamma})^{hyp}\end{array}\right\}$$
where the target can be decomposed as:
$$\left\{ \begin{array}{l}\hspace{.1cm} \mbox{cusps of} \\ D({P}_{\widehat\gamma})^{hyp}\end{array}\right\}=\left\{ \begin{array}{l}\hspace{.3cm}\mbox{splitting tori of the} \\ \mbox{JSJ-decomposition of}  \\ \hspace{1.3cm} D({P}_{\bar\gamma})\end{array}\right\}\amalg\left\{ \begin{array}{l}\hspace{.8cm}\mbox{cusp in}  \\   D({P}_{\widehat\gamma}) \cap  D({P}_{\widehat\gamma})^{hyp}\end{array}\right\}$$ 
The function $\varphi$ is defined as follows: if the cusps in $D({P}_{\widehat\gamma})$ are induced by the $\widehat\gamma$-arc in $PT^{1}P$ belonging to the characteristic sub-manifold of $D({P}_{\widehat\gamma}),$  $\varphi$ maps it to a splitting tori connecting the hyperbolic piece with the component of the characteristic sub-manifold where it is contained. Otherwise, the cusp belongs to $D({P}_{\widehat\gamma})^{hyp}$ and $\varphi$ sends it to itself, see Figure \ref{DPgJSJ}. 
Assume that there are more isotopy classes of $\widehat\gamma$-arcs in $PT^{1}P$ than the number of cusps of $D({P}_{\widehat\gamma})^{hyp}$. Then, there are two tori, associated with non-isotopic $\widehat\gamma$-arcs in $PT^{1}P,$ that belong to the same connected component of the characteristic sub-manifold. Since each component of the characteristic sub-manifold is a Seifert-fibered space over a punctured surface we have that all such arcs correspond to regular fibres. Thus, they are isotopic in the corresponding component hence isotopic in $PT^{1}P,$ contradicting the fact that they were not isotopic.   
\vskip .2cm
Finally, since two isotopic $\widehat\gamma$-arcs in $PT^1P$ induce a homotopy between their projections in $P.$ Then for the case for $D({P}_{\widehat\gamma})^{hyp}\neq\emptyset,$ we have that\string:
 $$ v_3\sharp\{\mbox{homotopy classes of} \hspace{.2cm}  \gamma\mbox{-arcs in} \hspace{.2cm} P\}\leq v_3\sharp\{\mbox{cusps of}  \hspace{.2cm} D({P}_{\widehat\gamma})^{hyp}\}.$$
 \end{proof}

Proceeding as in Theorem \ref{1}, we give a lower bound for the simplicial volume of any continuous lift complement over a closed geodesic on a hyperbolic surface\string: 

\begin{corollary}\label{1.2}
Given a pants decomposition $\Pi$ on a hyperbolic surface $\Sigma$  and $\gamma$ a closed geodesic, we have that for any continuous lift $\widetilde\gamma$ we have\string:
\[
\pushQED{\qed}  \|M_{\widetilde\gamma}\|  \geq \frac{1}{2}\sum_{P \in \Pi}(\sharp\{\mbox{ isotopy classes of} \hspace{.2cm}  \widetilde\gamma\mbox{-arcs in} \hspace{.2cm} PT^1P\}-3).\qedhere
\popQED
\]    
 \end{corollary}

This result implies that there exist geodesics $\gamma$ on $\Sigma$ such that the $\Vol(M_{\widehat\gamma})$ can be as big as we want. Let us fix a pants decomposition on $\Sigma,$ then for any $N\in\mathbb{N}$ there exist a closed geodesic with at least $N$ homotopy classes of geodesic arcs in one pair of pants. This is constructed by taking $N$ non homotopic geodesic arcs in a pair of pants and linking them to form a filling closed geodesic on $\Sigma.$
\vskip .2cm
The lower  bound of the volume of  $M_{\widehat\gamma}$ obtained in Theorem \ref{1}, does not have control on the length of the geodesic, even if each homotopy class of $\gamma$-arcs contributes to the length of $\gamma.$ So a natural question is how big can the volume of $M_{\widehat\gamma}$ be when the length of the filling close geodesic $\gamma$ is bounded. In subsection \ref{4.2} we try to understand this relation.

\subsection{Coding filling geodesics on surfaces by splitting along a simple closed geodesic} \label{coding}

This subsection gives a method for describing closed geodesics on hyperbolic surfaces that intersect a given simple closed geodesic. This will be useful for the results in the subsequents sections because it also allows us to identify the homotopy classes of geodesic arcs generated by splitting along the simple closed geodesic.
\vskip .2cm
Let $\Sigma$ be a hyperbolic surface and fix $\beta$ be a separating (non-separating) simple closed geodesic on $\Sigma.$  Denote by $\Sigma^1$ the union of $\beta$ and one of the connected components of $\Sigma\setminus\beta$ and by $\Sigma^2$ the union of $\beta$ with the other connected component (resp. let $\Sigma^1$ be the complement of $\mathcal{N}_\beta$ a maximal normal embedded neighborhood of $\beta$ in $\Sigma$ whose boundary is a closed geodesic). Choose a point $p\in \beta$ to be the base point of the fundamental group of  $\Sigma$ (resp. $p$ in the self intersection of the closed geodesic which is the boundary of $\mathcal{N}_\beta$). As a consequence of Van Kampen's Theorem, when $\beta$ is separating, $\pi_1(\Sigma,p)$ is isomorphic to the free product of $\pi_1(\Sigma^1,p)$ with $\pi_1(\Sigma^2,p)$  amalgamated on the subgroup $\pi_1(\beta,p).$  If $\beta$ is non-separating, $\pi_1(\Sigma,p)$ is isomorphic to the HNN extension of $\pi_1(\Sigma^1,p)$ relative to an isomorphism between two simple curves in the boundary of $\mathcal{N}_\beta.$   Notice also that the simple closed geodesic $\tau$ inside $\mathcal{N}_\beta$ perpendicular to $\beta,$ passing thought $p,$ represents the stable element, denoted by $t.$

\begin{definition*} \normalfont 
Let $\beta$ be a separating (non-separating) simple closed geodesic on $\Sigma,$ and $n$ be a positive integer. A finite sequence 
$$(g_1,g_2,...,g_n)\hspace{.2cm} \hspace{.2cm}(\mbox{resp.} \hspace{.2cm} (g_0,t^{\varepsilon_1}, g_1,t^{\varepsilon_2},...,g_{n-1},t^{\varepsilon_n}))$$
of elements of $ \pi_1(\Sigma^1,p)\ast_\beta  \pi_1(\Sigma^2,p)$ (resp. $\pi_1(\Sigma^1,p)^{\ast_\beta})$ is  $\beta$\textit{-cyclically reduced} if the following conditions holds\string:
\begin{enumerate}
\item each $g_i$ is in one factor $\pi_1(\Sigma^1,p)$ or $\pi_1(\Sigma^2,p)$  (resp. $\varepsilon_i\in\{1,-1 \}$ and  $g_i\in\pi_1(\Sigma^1,p));$
\item for each $i\in \{1,...,n-1 \},$ $g_i$ and $g_{i+1}$ are not in the same factor, (resp. there is no consecutive subsequence of the form $(t^{\varepsilon}, g_i,t^{-\varepsilon})$ whose product is reducible);
\item if $n=1,$ then $g_1$ is not the identity (resp. for $n=0);$
\item any cyclic permutation of the sequence satisfies $(1)$-$(3).$
\end{enumerate}
\end{definition*}

The following statement allows, by fixing a minimal system of generators in $\pi_1(\Sigma^i,p),$ to associate to each oriented closed geodesic in $\Sigma$ a cyclically reduced sequence. The interested reader can find this material in \cite{Cha10}.

\begin{theorem*}[Chas \cite{Cha10} Section 2 and 4]
\begin{enumerate}
\item Let $s$ be a conjugacy class of $\pi_1(\Sigma^1,p)\ast_\beta  \pi_1(\Sigma^2,p)$ (resp. $\pi_1(\Sigma^1,p)^{\ast_\beta}).$ Then there is a representative of $s$ which can be written as a product\string: $$g_1g_2...g_n\hspace{.2cm} \hspace{.2cm}(\mbox{resp.} \hspace{.2cm} g_0t^{\varepsilon_1}g_2t^{\varepsilon_2}...g_{n-1}t^{\varepsilon_n}),$$ where $(g_1,g_2,...,g_n)$ (resp. $(g_0,t^{\varepsilon_1}, g_1,t^{\varepsilon_2},...,g_{n-1},t^{\varepsilon_n})$) is a  $\beta$-cyclically reduced sequence.
\item If $n$ is a positive integer and $$(g_1,g_2,...,g_n)\hspace{.2cm}(\mbox{resp.} \hspace{.2cm} (g_0,t^{\varepsilon_1}, g_1,t^{\varepsilon_2},...,g_{n-1},t^{\varepsilon_n}))$$ 

is $\beta$-cyclically reduced sequence then the product 

$$g_1g_2...g_n\hspace{.2cm}(\mbox{resp.} \hspace{.2cm} g_1t^{\varepsilon_1}g_2t^{\varepsilon_2}...g_{n}t^{\varepsilon_n})$$  is not the identity.
\item For any pair of $\beta$-cyclically reduced sequences representing the same conjugacy class in $\pi_1(\Sigma^1,p)\ast_\beta  \pi_1(\Sigma^2,p)$ (resp. $\pi_1(\Sigma^1,p)^{\ast_\beta}).$ The number of terms of the two sequences is the same.
\end{enumerate}
\end{theorem*}

Consider $\beta$ any essential simple closed geodesic in  $\Sigma_{1,1}$ and the HNN extension of $\pi_1(\Sigma_{1,1},p)$ obtained by splitting $\Sigma_{1,1}$ along $\beta$\string: $$\langle a,b,t\mid tat^{-1}=b \rangle,$$ where $ab^{-1}$ represents the conjugacy class of a simple closed curve parallel to the cusp, $t$ the stable element representing the class of $\tau$ and $b$ representing the class of $\beta.$ 

\begin{corollary}\label{arcs}
Let $\gamma$ be an oriented closed geodesic in $\Sigma_{1,1}$ intersecting $\beta,$ with its $\beta$-cyclically reduced sequence of the form $(g_0,t, g_1,t,...,g_{n-1},t)$ such that each $g_i$  starts with $b^{\pm1}$ and ends with $a^{\pm1}.$ Then the sequence of $\gamma$-arcs $(\gamma\mid_{I_1},\gamma\mid_{I_2},...,\gamma\mid_{I_n})$ that result from splitting $\gamma$  along $\beta$ satisfy the following properties\string:
\begin{enumerate}
\item the $\gamma$-arc $\gamma\mid_{I_i}$ is representative of $(t,g_i,t);$
\item the inclusion $\gamma\mid_{I_i}\subset \Sigma_{1,1}\setminus \beta$ holds;
\item if $\gamma\mid_{I_i}$ and $\gamma\mid_{I_j}$ are freely homotopic relative to $\beta,$ then $g_i=g_j.$
\end{enumerate}
\end{corollary}

\begin{proof}[\bf{Proof}] 
By Lemma 5.2 in \cite{Cha10} there exist sequence of curves $(\gamma_0,\gamma_1,...,\gamma_{n-1})$ such that the curve $\gamma_i\subset \Sigma^1$ is a representative of $g_i.$ 
\vskip .2cm
Denote by $\gamma$ the curve $\gamma_0\tau\gamma_1\tau...,\gamma_{n-1}\tau,$ which clearly is a representative of $\gamma$ and consider the $\gamma$-arc $\gamma\mid_{I_i}\subset \Sigma_{1,1}$ which is homotopic relative to the boundary to the concatenation of arcs $\tau_0,\gamma_i$ and $\tau_1,$ where $\tau_0$ and $\tau_1$ are the two $\tau$-arcs from the intersection point of $\tau$ with $\beta$ to $p.$
\vskip .2cm
To prove $3)$ it is enough to show that there is an isomorphism between $\pi_1(\Sigma^{1},p)$ and the group of free homotopy classes of arcs in $\Sigma_{1,1}\setminus \beta$ relative to the boundary, which start in one component and end in the other. This isomorphism is given as follows, given $\alpha$ a loop based at $p$ construct the arc $\tau_0*_p\alpha*_p\tau_1.$ The inverse of the isomorphism consist in collapsing the endpoints of a given arc to $p$ along $\tau_0$ and $\tau_1.$
\end{proof}

 \begin{figure}[h]
\centering
\includegraphics[scale=0.4]{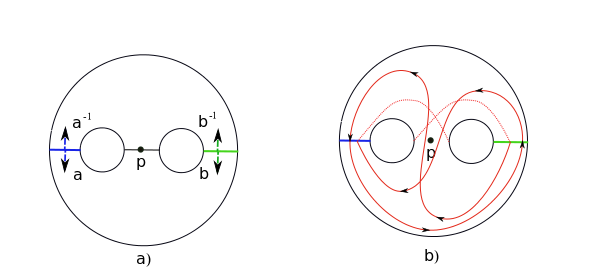}
\caption{ a) Orientation of the arcs transversal to the seam edges. b) Geodesic  arc corresponding to the code $(t,bab^{-1}a^{-1},t).$
}\label{1arcs}
\end{figure} 

\begin{note}\label{arcs0}
The $\gamma$-arcs $\gamma\mid_{I_i}$ of Corollary \ref{arcs} can be encoded by using two oriented seam geodesic arcs between $\beta$ and the boundary of $\Sigma_{1,1}.$ The coding comes from replacing each letter in $g_i$ with an arc connecting one seam edge to another following the direction prescribed by $g_i$ (see Figure \ref{1arcs}).
\end{note}

In the case when $\beta$ is a separating  simple closed geodesic on a hyperbolic surface $\Sigma,$ by (\cite{Cha10}, Lemma 3.2) we have the following result\string:

\begin{corollary}\label{arcs2}
Let $\gamma$ be an oriented closed geodesic in $\Sigma$ intersecting the separating geodesic $\beta.$ Consider its $\beta$-cyclically reduced sequence of the form $(g_1,g_2,...,g_n).$ Then the sequence of $\gamma$-arcs $(\gamma\mid_{I_1},\gamma\mid_{I_2},...,\gamma\mid_{I_n})$ that results from splitting $\gamma$  along $\beta$ satisfy that if we identify the endpoint of an arc $\gamma\mid_{I_i}$ along $\beta,$ then the resulting loop is representative of the conjugacy class of $g_i.$\hfill $\square$
\end{corollary}

The following result is a criterion to verify when the unique $\beta$-cyclically reduced sequence associated to $\gamma$ in Corollary \ref{arcs} is a filling geodesic.

\begin{lemma}\label{fitor}
Let $\gamma$ be an oriented closed geodesic as  in Corollary \ref{arcs}. If there exist a term $g_i,$ in the $\beta$-cyclically reduced sequence of $\gamma,$ with a sub word of the form $a^{\pm1}b^{\pm1}$ or $b^{\pm1}a^{\pm1}.$ Then $\gamma$ is filling.
\end{lemma}
\begin{proof}[\bf{Proof}] 
It is enough to show that there exist two simple loops formed by sub arcs $\alpha_1$ and $\alpha_2$ of $\gamma$ such that $i(\alpha_1,\alpha_2)\neq 0.$  
\vskip .2cm
Since $\gamma$  intersects $\beta$ there exist $g_i\notin  \langle a \rangle\cup\langle b \rangle.$ Then $\gamma$ is not a simple closed curve in $\Sigma_{1,1}$ and there is a simple loop formed by a sub arc $\alpha_1$ of $\gamma$ which is transversal to $\beta$ ( \cite{Cha10}, Theorem 5.3).
\vskip .2cm
By the Remark \ref{arcs0}, the fact that there exist $g_i$ with a sub word of the form $a^{\pm1}b^{\pm1}$ or $b^{\pm1}a^{\pm1}$ implies that there is a $\gamma$-arc $\gamma\mid_{I_i}$ corresponding to $g_i$ (Corollary \ref{arcs}) that contains a sub arc $\alpha_2$ which induces a loop homotopic to $\beta.$ 
\end{proof}

A first application to the previous Lemma \ref{fitor}  and Theorem \ref{1}, is the following result.

\begin{reptheorem}{mod}
For the modular surface $\Sigma_{mod},$ there exists a sequence $\{\gamma_k\}$ of filling closed geodesics on $\Sigma_{mod}$ with $n_{\gamma_k}\nearrow \infty$ such that,
$$ \Vol(M_{\widehat{\gamma_k}})\geq v_3 \frac{n_{\gamma_k}}{12},$$
where $n_{\gamma_k}$ is half the period of the continued fraction expansion of $\gamma_k$ and $v_3$ is the volume of a regular ideal tetrahedron. 
\end{reptheorem}

\begin{proof}[\bf{Proof}] 
Let $x=\begin{pmatrix} 
1 & 1 \\
0 & 1 
\end{pmatrix}$ 
and $y=\begin{pmatrix} 
1 & 0 \\
1 & 1 
\end{pmatrix}.$  
Consider the once-punctured modular torus, whose fundamental group is generated by\string:
$$t:=xy=\begin{pmatrix} 
2 & 1 \\
1 & 1 
\end{pmatrix}\hspace{.5cm}\mbox{and}  \hspace{.5cm}a:=x^{-1}y^{-1}=\begin{pmatrix} 
2 & -1 \\
-1 & 1 
\end{pmatrix},$$
which has the following presentation $\langle a,b,t\mid tat^{-1}=b \rangle.$ 
Let us denote as $q$ the covering map of the once-punctured modular torus over the modular surface. Notice that $q$ is a finite cover of index $6.$ Also, the geodesic representative of $b$ and $a$ on the once-punctured modular torus are the same because they are conjugated. Moreover,  $ba^{-1}$ represents the conjugacy class of a simple loop that is retractable to the cusp of the once-punctured modular torus, because $ba^{-1}=tat^{-1}a^{-1}$ (the conmutator of $a$ and $t)$ is equal to the parabolic matrix $-x^6.$ 

 \begin{figure}[h]
\includegraphics[scale=0.35] {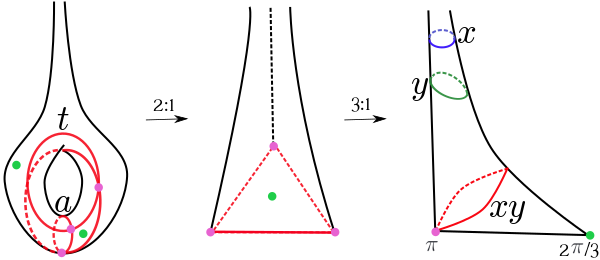}
\caption{The once-punctured modular torus and the three-fold covers of the modular surface, with loops representing the parabolic matrices $x$ and $y;$ the closed geodesic representing $xy$ is marked on the modular surface. Also, we marked the three closed geodesics lifts of $xy$ on the once-punctured modular torus, two of which are $a$ and $t.$
}\label{loop}
\end{figure} 

By using the coding of closed geodesics in the modular surface with respect $\{x,y\}$ (see \cite{BPS17}, Section 3), we consider the sequence of closed geodesics $\{\gamma_k\}$ on the modular surface $\Sigma_{mod}$ such that $\gamma_k$ is the closed geodesic representative of  the matrix\string:
$$A_k:=\prod\limits_{i=1}^{k}(x^{6i+1}y).$$
 And let $\{\overline {\gamma_k}\}$ be the sequence of closed geodesics on the once-punctured modular torus such that $\overline {\gamma_k}$ is the closed geodesic representative of the matrix\string:
$$B_k:=\prod\limits_{i=1}^{k}((ba^{-1})^{i}t).$$

Notice that $\overline {\gamma_k}$ is a lift of $\gamma_k$ under the covering map $q.$ To see this it is enought to verify that both closed geodesic lift to the same geodesic line in the universal cover, which is a consequence of the fact that $A_k=B_k$ up to sign. Indeed, we can rewrite each factor of $A_k$ in terms of $\{a,b^{-1},t\}$ by the following equality\string:
$$(ba^{-1})^it=(tat^{-1}a^{-1})^it=(x^6)^ixy=x^{6i+1}y.$$
Moreover, $B_k$ represents an $a$-cyclic reduced word associated to $\overline {\gamma_k}$ on the once-punctured modular torus, meaning that $\overline {\gamma_k}$ is filling by Lemma \ref{fitor}. Also, if we split the once-punctured modular torus along the simple closed geodesic associated to $a,$ then the number of homotopy classes of $\overline{\gamma_k}$-arcs after the splitting  is $k,$ by Corollary \ref{arcs}.
\vskip .2cm
In order to estimate a lower bound for the volume, let $\overline {M}$ be the projective unit tangent bundle of once-punctured modular torus and let $\overline {\Gamma_k}$ be all the lifts of $\gamma_k$ under the covering map $q.$ As $\overline {\gamma_k}$ is filling, then $\overline{M}_{\widehat{\overline {\gamma_k}}}$ is hyperbolic. Moreover, as simplicial volume is multiplicative under finite coverings and does not increase under Dehn filling (\cite{Thu79} Theorem 6.5.4), then\string:

 $$ 6\Vol(M_{\widehat{\gamma_k}})=6\|M_{\widehat{\gamma_k}}\|=\| \overline{M}_{\widehat{\overline {\Gamma_k}}}\|\geq \| \overline{M}_{\widehat{\overline {\gamma_k}}}\|=\Vol(\overline{M}_{\widehat{\overline {\gamma_k}}}).$$

Finally, by Theorem \ref{1} we have a lower bound of $\Vol(\overline{M}_{\widehat{\overline {\gamma_k}}})$ given by $\frac{v_3}{2}k,$ which implies that $v_3\frac{k}{12}$ is a lower bound for $\Vol(M_{\widehat{\gamma_k}}). $

\end{proof}

\begin{note*}
For the filling closed geodesics  $\{\gamma_k\}$ on Corollary \ref{mod} we can also rewrite the combinatorial upper bound found by Bergeron, Pinsky and Silberman (\cite{BPS16}, Theorem 3.5) for the volume of the canonical lift complement of $\widehat{\gamma_k}$ in terms of $n_{\gamma_k},$ as
$$C n_{\gamma_k}\ln(n_{\gamma_k}),$$ 
where $n_{\gamma_k}=k+1$ and $C$ a positive constant that depends on $\Sigma_{mod}.$
\end{note*}

\subsection{Lower bound of the volume of $M_{\widehat\gamma}$ in terms of the length of $\gamma$}\label{4.2}

Given $P$ a pair of pants in a pants decomposition of a hyperbolic surface, we construct all possible non homotopic geodesic arcs on $P$ whose length is bounded by some constant and then induce a filling geodesic which contains all this arcs. In this case, we have rewritten the lower bound for the volume of the canonical lift complement given by Theorem \ref{1} in terms of the length of the geodesic.

\begin{theorem}\label{pi-b}
Given a hyperbolic metric $X$ on the surface $\Sigma_{1,1},$ there is a sequence $\{\gamma_n\}$ of filling closed geodesics on $\Sigma$ with $\ell_{X}(\gamma_n)\nearrow \infty$ such that,
$$\Vol(M_{\widehat{\gamma_n}}) \geq k_X\frac{ \ell_{X}(\gamma_n)}{\ln(\ell_{X}(\gamma_n))},$$
where $k_X$ depends on the hyperbolic metric $X.$
 \end{theorem}

\begin{proof}[\bf{Proof}] 
For the sake of concreteness, we will start proving the result for a particular hyperbolic metric on $\Sigma_{1,1}.$  Fix the following representation $\fun{\rho}{\pi_1(\Sigma_{1,1}):=\langle a,t\rangle}{\PSL_2(\mathbb{R}) }$ such that
$$\small {\rho(a)=A=\left(\begin{array}{cc}
\sqrt{2}&1+\sqrt{2}\\
-1+\sqrt{2}&\sqrt{2}
\end{array}\right) 
\hspace{.5cm}\mbox{and} \hspace{.5cm} 
\rho(t)=T=\left(\begin{array}{cc}
-1+\sqrt{2}&0\\
0&1+\sqrt{2}
\end{array}\right).}
$$
Consider the splitting of $\Sigma_{1,1}$ along a simple close  geodesic associated to $A,$ then we have the HNN extension of $\pi_1(\Sigma_{1,1},p)=\langle a,b,t\mid tat^{-1}=b \rangle,$ with\string:
$$\rho(b)=TAT^{-1}=\left(\begin{array}{cc}
\sqrt{2}&-1+\sqrt{2}\\
1+\sqrt{2}&\sqrt{2}
\end{array}\right).
$$
Notice that  $ab^{-1}$ is freely homotopic to a simple loop that is retractable to the cusp of $\Sigma_{1,1}.$  
\vskip .2cm
For each $n\in\mathbb{N}$ we define the following $A$-cyclically reduced sequence:
$$(g_1,t,g_2,t,...,t,g_{12\cdot(3^{n}-1)},t)\hspace{.2cm}\mbox{where} \hspace{.2cm} g_i \in \langle a,b\rangle.$$
Moreover  $\{g_i\}_{i=1}^{12\cdot(3^{n}-1)}$ is the set of different reduced words in $\langle a,b\rangle$ starting with $b$ or $b^{-1},$ and ending with $a$ or $a^{-1},$ which have word length at must $n$ and at least $4.$ 
\vskip .2cm
By Corollary \ref{arcs}, the number of homotopy classes of arcs in each $n^{th}$-se\-quence is equal to $\sharp\{g_i\}_{i=1}^{12\cdot(3^{n}-1)}=12\cdot(3^{n}-1).$  Let $\gamma_n$ be the unique geodesic associated to the product of the $n^{th}$-sequence. Each one of the geodesics $\gamma_n$ belongs to different mapping class group orbits, because in this case the self-intersection number of  $\gamma_n$ is bounded from below by the number of terms of the sequence, which grows exponentially. Therefore, by Lemma \ref{fitor}, $\{\gamma_n\}$ is a sequence of filling closed geodesics on $\Sigma_{1,1}$ with $\ell_{X_0}(\gamma_n)\nearrow \infty.$ 
\vskip .2cm
By Theorem \ref{1}, we have that\string:
\begin{equation}\label{in2}
v_3\cdot 6\cdot(3^{n}-1)= \frac{v_3}{2} \sharp\left\{\mbox{homotopy clases of} \hspace{.1cm}  \gamma_n\mbox{-arcs}\right\}\leq \Vol(M_{\widehat{\gamma_n}}). 
 \end{equation}
The last part of the proof consist in rewriting the combinatorial lower bound in terms of the length of $\gamma_n.$ To do so, for each $\gamma_n,$ let $L_n$ be such that\string:
$$n=\frac {\ln(\ln(3)L_n )-\ln(\ln(\ln(3)L_n))-\ln(3)}{\ln(3)}.$$
Using the trace of the matrices we can calculate the length of the geodesics associated to $ab$ and $t.$ In this way we give the following upper bound of the length of $\gamma_n$ by using the word length with respect the generating set $\{a,b,t\}$\string:

$$\ell_{X_0}(\gamma_n)\leq 1.15\cdot|g_1tg_2t...tg_{12\cdot(3^{n}-1)}t| \leq 1.15\cdot12\cdot3^{n}(n+3)\leq 14\cdot L_n.$$
Then,
$$
6\cdot(3^{n}-1)\geq\frac{\ln(3)L_n} {\ln(\ln(3)L_n)}.
$$
Finally, by inequality (\ref{in2}) we have that:
$$v_3\frac{ \ell_{X_0}(\gamma_n)}{\ln(\ell_{X_0}(\gamma_n))}\leq \Vol(M_{\widehat{\gamma_n}}).$$
The proof of this result for any hyperbolic metric, follows from the fact that any pair of hyperbolic metrics on a hyperbolic surface are bi-Lipschitz  (see for example \cite{BPS16}, Lemma 4.1).
\end{proof}

To generalize  Theorem \ref{pi-b} to any hyperbolic surface $\Sigma,$ consider the following method for extending a given closed geodesic that fills a subsurface of $\Sigma$ to a filling closed geodesic on $\Sigma.$  Let $\beta$ be a separating oriented simple closed geodesic on $\Sigma$ and denote $\Sigma\setminus \beta$ by $\Sigma^1\coprod \Sigma^2.$ 
\vskip .2cm
Given two closed geodesics $\alpha_1$ and $\alpha_2$ on $\Sigma^1$ and $\Sigma^2$ respectively, choose a minimal length common perpendicular segment $\eta$ between $\alpha_1$ and $\alpha_2,$  with extremal points $p_1$ and $p_2.$ Let $\alpha_1\star\alpha_2$ be the closed geodesic in the homotopy class of the piecewise closed geodesic $\bar\alpha_{1,2}$ that travels from $p_1$ along $\alpha_1$ on the direction given by $\beta,$ then follows $\eta,$ continues from $p_2$ along $\alpha_2$ on the same direction given by $\beta$ and finishes by $\eta^{-1}.$\vskip .2cm
 \begin{figure}[h]
\centering
\includegraphics[scale=0.5] {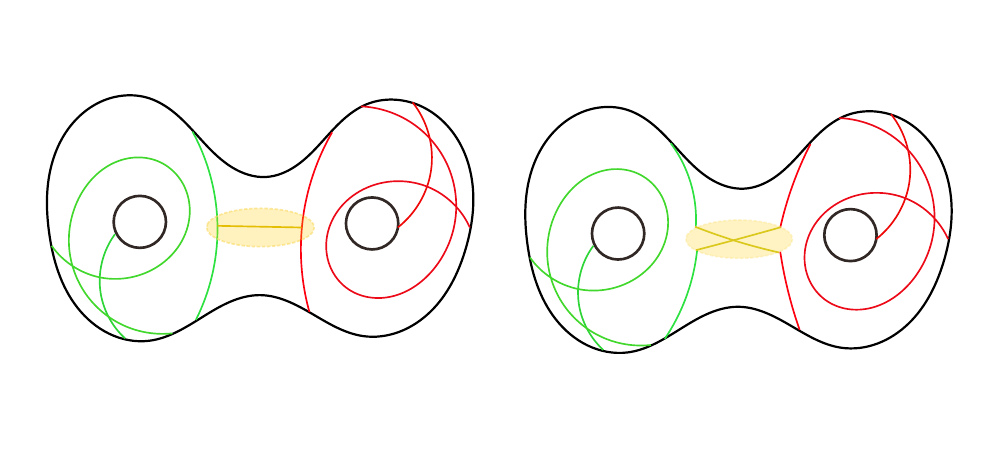}
\caption{ From $\bar\alpha_{1,2}$ to $\bar{\alpha_1\star\alpha_2}.$
}\label{geo1}
\end{figure} 
As $\bar\alpha_{1,2}$ is not self-transverse, we will slightly modify it to obtain one that is.
\vskip .2cm
 Take a $\delta$ neighborhood of ${\eta}$  such that it does not contain self-intersection of $\alpha_1$ and $\alpha_2.$ Erase the segment of ${\bar\alpha_{1,2}}$  that is contained the $\delta$ neighborhood of the ${\eta}$ geodesic segment, and link the consecutive extremal points in the same $\delta$ neighborhood by a pair of intersecting minimal length geodesic segments linking the remaining segments of ${\bar\alpha_{1,2}}.$ This piecewise geodesic, denoted by $\bar{\alpha_1\star\alpha_2},$ has no self-tangency points.\vskip .2cm

 \begin{figure}[h]
\centering
\includegraphics[scale=0.5] {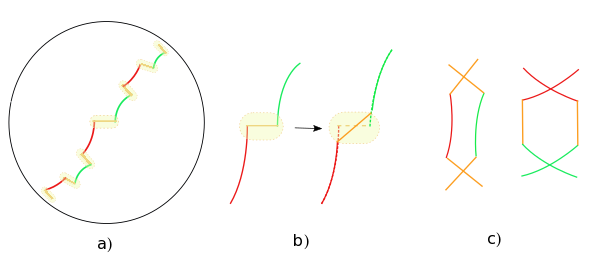}
\caption{a) A lift of $\bar\alpha_{1,2}$ in $\mathbb{H}^2.$ \hspace{.2cm} b) The modification from $\widetilde{\bar\alpha_{1,2}}$ to $\widetilde{\bar{\alpha_1\star\alpha_2}}$ by erasing the segments $\widetilde{\eta}$ and re-gluing.  \hspace{.2cm} c) Possible bigons given by the intersection of two lifts of  $\bar{\alpha_1\star\alpha_2}.$ 
}\label{pd2}
\end{figure}
\begin{claim}\label{mia}
$\bar{\alpha_1\star\alpha_2}$ has minimal self-intersection.
\end{claim}

\begin{proof}[Proof] By Theorem \ref{2-gone}, we just need to show that $\bar{\alpha_1\star\alpha_2}$ has no singular $1$-gons nor $2$-gons. Indeed, take the lifts of $\bar{\alpha_1\star\alpha_2}$ in the universal cover of $\Sigma,$  which are piecewise geodesic lines. It is enough to show that every pair of them intersects at most once. Certainly, if there were two consecutive intersections they cannot happen in segments of the same geodesic segment, so in-between there must be at least one gluing point of two distinct geodesic segments. Moreover, from the fact that the angle on the gluing points is obtuse, there most be at least two gluing points in each $2$-gon side. If that is the case, we would have geodesic side hexagons whose sum of interior angles is bigger or equal to $4\pi,$ which is a contradiction in the hyperbolic plane.
\end{proof}
By rounding the four corners of $\bar{\alpha_1\star\alpha_2},$ we get a self-transverse closed curve homotopically transversal to $\alpha_1\star\alpha_2,$ in particular this means that $M_{\widehat{\alpha_1\star\alpha_2}}\cong M_{\widehat{\bar{\alpha_1\star\alpha_2}}}.$

\begin{claim}\label{figl}
If each $\alpha_i$ is a filling closed geodesic on each $\Sigma^i$, then $\alpha_1\star\alpha_2$ is a filling closed geodesic on $\Sigma.$ 
\end{claim}

\begin{proof}[Proof] Notice that $\Sigma\setminus(\alpha_1\cup\alpha_2)$ differs from  $\Sigma\setminus(\bar{\alpha_1\star\alpha_2})$  only in the regions whose sides contain $p_1$ and $p_2.$ And $\bar{\alpha_1\star\alpha_2}$ splits the ring around $\beta$ in the original partition and deforms the other two adjacent regions, preserving their homeomorphism type. Moreover, as the property of being filling is preserved by transversal homotopy (Corollary \ref{comp}), then $\alpha_1\star\alpha_2$ is also filling.
\end{proof}

\begin{reptheorem}{pi-b-2}
Given a hyperbolic metric $X$ on a surface $\Sigma,$ there exist a sequence $\{\gamma_n\}$ of filling closed geodesics on $\Sigma$ with $\ell_X(\gamma_n)\nearrow \infty$ such that,
$$ \Vol(M_{\widehat{\gamma_n}})\geq c_X\frac{ \ell_X(\gamma_n)}{\ln(\ell_X(\gamma_n))},$$
where $c_X$ depends on the hyperbolic metric $X.$
 \end{reptheorem}

\begin{proof}[\bf{Proof}] 
If the surface $\Sigma$ contains a once-punctured torus, choose a separating closed geodesic $\beta$ such that $\Sigma\setminus \beta:=\Sigma^1\coprod \Sigma^2,$ where $\Sigma^1$ is homeomorphic to $\Sigma_{1,1}.$ Choose a pair of simple closed geodesics $\alpha$ and $\tau$ on $\Sigma^1$ such that $i(\alpha,\tau)=1.$ Consider the closed geodesics $\alpha_n$ on $\Sigma^1$ homotopic to the closed curves defined in Theorem \ref{pi-b}. Extend $\{\alpha,\beta\}$ to a pants decomposition $\Pi$ on $\Sigma.$ And fix $\eta_0$ a filling closed geodesic on $\Sigma^2.$ Then we construct $\gamma_n$ the filling closed geodesics (Claim \ref{figl}) associated to\string:
$$\alpha_n\star\eta_0.$$
By Theorem \ref{1} we have that\string:
\begin{equation}\label{in1}
v_3\cdot6\cdot(3^{n}-1)= \frac{v_3}{2}  \sum_{P\in\Pi}\sharp\{\mbox{homotopy classes of} \hspace{.2cm}  \gamma\mbox{-arcs in} \hspace{.2cm} P\}\leq \Vol(M_{\widehat{\gamma_n}}). 
 \end{equation}
The last part of the proof consist in rewriting the combinatorial lower bound in terms of the length of $\gamma_n.$ By using the word length with respect the generating set 
$\{\alpha,\tau\}$ we can find constants  $c_X$ and $\varepsilon_X$ that only depends on the metric such that\string:
$$c_X\frac{ \ell_X(\gamma_n)}{\ln(\ell_X(\gamma_n))} \leq k_X\frac{ \ell_X(\gamma_n)-\varepsilon_X}{\ln(\ell_X(\gamma_n))} \leq k_X\frac{ \ell_X(\alpha_n)}{\ln(\ell_X(\alpha_n))} \leq v_3\cdot 6\cdot(3^{n}-1).$$
Finally, by the inequality (\ref{in1}) we have that\string:
$$c_X\frac{ \ell_X(\gamma_n)}{\ln(\ell_X(\gamma_n))} \leq \Vol(M_{\widehat{\gamma_n}}).$$
In the case where $\Sigma$ is a $n$-punctured sphere, we consider $\pi:\Sigma'\rightarrow \Sigma$ a finite cover where $\Sigma'$ contains a once-punctured torus, and the induced covering $\widehat\pi$ from $PT^1\Sigma'$ to $PT^1\Sigma.$ Consider the previous geodesics $\{\gamma_n\}$  on $\Sigma'$ and $\{\pi(\gamma_n)\}$ in $\Sigma.$ Then the volume of the complement of $\widehat{\pi(\gamma_n)}$ on $PT^1\Sigma$ is a constant (given by the degree of the covering) times the volume of the complement of the canonical lifts of all the translates of $\gamma_n$ under the group of deck transformations on $PT^1\Sigma',$  which is grater than the volume corresponding to the complement of $\widehat\gamma_n$ on $PT^1\Sigma'.$ So the filling closed geodesics $\{\pi(\gamma_n)\}$  in $\Sigma$ is the sequence wanted.
\end{proof}

\section{Further comments}\label{s4}

As noticed by Foulon and Hasselblatt in  (\cite{FH13}, Theorem 1.12) any continuous lift complement of a filling closed geodesic $\gamma$ is a hyperbolic $3$-manifold of finite volume. This volume is bounded from below by the isotopy classes of $\widetilde\gamma$-arcs in the unit tangent bundle over a given pants decomposition of the surface (Corollary \ref{1.2}). This gives rise to new questions, for example,  how the volume complement varies from the canonical lift to other lifts in $PT^1\Sigma$ over the same closed geodesic.
\vskip .2cm
We start by showing that given a hyperbolic metric on a hyperbolic surface we can construct a sequence of closed geodesics and their respective non-canonical lifts, whose volume complement has a lower bound which is linear in the length of the geodesic.

\begin{proposition}\label{lin}
For any hyperbolic metric $X$ on $\Sigma,$ there exist a sequence of $\{\gamma_n\}$ filling closed geodesics and respective lifts $\{\widetilde{\gamma_n}\}$ in  $PT^1\Sigma$ with $\ell_X(\gamma_n)\nearrow \infty,$ such that,
$$ k_X\ell_X(\gamma_n) \leq  \Vol(M_{\widetilde{\gamma_n}}),$$
where $k_X$ is a positive constant that depends on the metric $X.$ Moreover, there exists a constant $V_0>0$ such that for the canonical lift $\Vol(M_{\widehat{\gamma_n}})< V_0$ for every $n\in\mathbb{N}.$
 \end{proposition}
 
By using techniques found in \cite{CKP16}, this result could maybe be improved by constructing a sequence of non-canonical lifts, whose volume complement has a lower bound which is quadratic in the length of the geodesic. Motivated by this, we conjecture that the volume of the canonical lift complement minimizes the volume of the complement over all the possible lifts of a fixed filling closed geodesic on the surface.

\begin{conjecture}
Given $\gamma$ a filling closed geodesic on a hyperbolic surface, then
$$\Vol(M_{\widehat{\gamma}})=\min_{\widetilde\gamma  \rightarrow \gamma  }\{\Vol(M_{\widetilde{\gamma}})\}.$$
\end{conjecture}
 
\begin{proof}[\bf{Proof of Proposition \ref{lin}}] 
We will start proving the statement for the case $\Sigma=\Sigma_{1,2}.$ We start by splitting $\Sigma_{1,2}$ along a simple closed separating geodesic into pieces homeomorphic to $\Sigma_{1,1}$ and $\Sigma_{0,3}.$ Choose a base point $p$ and we write $\pi_1(\Sigma_{1,2},p)$ as the amalgamated product along the previous separating simple geodesic, then $\pi_1(\Sigma_{1,1},p)\ast_\mathbb{Z} \pi_1(\Sigma_{0,3},p)=\langle a_1,t_1 \rangle\ast_\mathbb{Z}\langle a_2,b_2 \rangle.$ 
\vskip .2cm
The first step consists on constructing the family of closed filling geodesics. For each $n\in\mathbb{N}$ we define the following cyclically reduced sequence\string:
$$(t_1^{n} a_1t_1a_1t_1^{-1},b_2,a_1,b_2,...,a_1,b_2)\hspace{.2cm}\mbox{which has length} \hspace{.2cm}2\lceil\ln(n)\rceil, $$
let $\gamma_n$ be the unique geodesic associated to the product of this sequence.
\vskip .2cm
By Corollary \ref{arcs2}, consider the sequence of subintervals $I_t,$ $\{I_{a_i}\}_{i=1}^{\lceil\ln(n)\rceil-1}$ and $\{I_{b_i}\}_{i=1}^{\lceil\ln(n)\rceil}$  of $[0,\ell_{X} (\gamma_n)],$ such that $\gamma_n\mid_{I_t},$ $\gamma_n\mid_{I_{a_i}}$ and $\gamma_n\mid_{I_{b_i}}$ are respectively homotopic to $t_1^{n} a_1t_1a_1t_1^{-1},$  $a_1$ and $b_2$ relative to the boundary in $\Sigma_{1,1}.$ Since the closed geodesics associated to $a_1$ and $t_1$ intersect each other, this implies that each $\gamma_n\mid_{I_{a_i}}$ intersects  $\gamma_n\mid_{I_t}$ $n\pm 1$ times. Notice that $\gamma_n$ is filling because by Lemma \ref{fitor} $\gamma_n\mid_{I_t}$  is filling in $\Sigma_{1,1}$ and the piece corresponding to $\Sigma_{0,3}$ is split by $\gamma_n\mid_{I_{b_i}}.$ We now construct the non-canonical lifts $\widetilde{\gamma_n}$ by modifying the canonical lift $\widehat{\gamma_n}$ in the pre-image of $\gamma_n\mid_{I_t}$ under the map ${PT^1\Sigma}\rightarrow{\Sigma}.$ 
\vskip .2cm
We start by splitting the sub arc $\gamma_n\mid_{I_t}$ by the closed simple geodesic associated to $a_1$ in $\{ \gamma_n\mid_{I_{t_j}} \}_{j=0}^{n}$  sub arcs. As there is an injection between  the number of sub arcs $\{ \gamma_n\mid_{I_{t_j}} \}_{j=0}^{n}$ and the set $\{0,1\}^{\lceil \ln (n)\rceil },$ then to each sub arc we can associate a unique sequence. 
\vskip .2cm
Let us take the pre-image of $\gamma_n\mid_{I_{t_j}}$ in $PT^1\Sigma_{1,2}$ under the projection map ${PT^1\Sigma}\rightarrow{\Sigma}.$  This is an annulus that contains $\widehat{\gamma_n}\mid_{I_{t_j}}$ and $\lceil \ln (n)\rceil$ punctures corresponding to the intersection of $\gamma_n\mid_{I_{a_i}}$ with $\gamma_n\mid_{I_{t_j}}.$ If the sequence associated to $\gamma_n\mid_{I_{t_j}}$ has $0$ in the $k$-coordinate we  do not modify the lift $\widehat{\gamma_n}\mid_{I_{t_j}}$ in the fiber corresponding to the intersection with the ${\gamma_n}\mid_{I_{a_k}}.$ If it has $1$ in the $k$-coordinate then we modify $\widehat{\gamma_n}\mid_{I_{t_j}}$  in such way that the new lift $\widetilde{\gamma_n}\mid_{I_{t_j}}$  goes around the puncture generated by $\widehat{\gamma_n}\mid_{I_{a_k}}.$

\begin{figure}[h]
\centering
\includegraphics[scale=0.5] {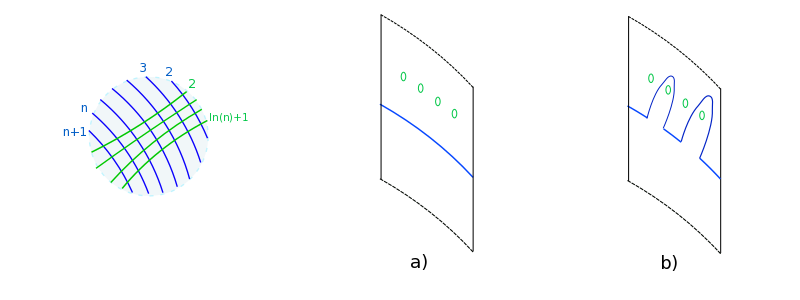}
\caption{a) The lift $\widehat{\gamma_n}\mid_{I_{t_j}},$  in the pre-image of $\gamma_n\mid_{I_{t_j}},$ and the $\gamma_n\mid_{I_{a_i}}$'s punctures.\hspace{.2cm} b) The new lift $\widetilde{\gamma_n}\mid_{I_{t_j}},$ in the pre-image of $\gamma_n\mid_{I_{t_j}},$  associated to the sequence $(0,1,0,1).$
}\label{li}
\end{figure}

\begin{claim}\label{ai}
Any pair of distinct sub arcs $\widetilde{\gamma_n}\mid_{I_{t_j}}$ and $\widetilde{\gamma_n}\mid_{I_{t_i}}$  are not ambient isotopic in $$(X,\partial X):=( T^1P\setminus  \{\widehat{\gamma_n}\mid_{I_{a_k}}\}_{k=1}^{\lceil \ln(n)\rceil-1}, \partial (T^1P\setminus  \{\widehat{\gamma_n}\mid_{I_{a_k}}\}_{k=1}^{\lceil \ln(n)\rceil-1})),$$
where $P$ is the pair of pants after splitting the piece $\Sigma_{1,1}$ by the simple closed geodesic associated to $a_1.$
\end{claim}

\begin{proof}[\bf{Proof}] 
Take $\gamma_n\mid_{I_{t_i}}$ and $\gamma_n\mid_{I_{t_j}}$ with $i\neq  j,$ and their respective sequence in $\{0,1\}^{\lceil \ln (n)\rceil },$ then there exists a first coordinate $k$ where the two sequence do not coincide. Consider an embedded surface $S_k$ in $T^1P\setminus  \{\widehat{\gamma_n}\mid_{I_{a_k}}\}_{k=1}^{\lceil \ln(n)\rceil}$ bounding the arcs $\widehat{\gamma_n}\mid_{I_{a_k}}$ and $\widehat{\gamma_n}\mid_{I_{a_{k+1}}}.$
\vskip .2cm
By the cup product in $H^*(X;\mathbb{Z})$ and Lefschetz duality we have the following paring:
$$H_2(X,\partial X;\mathbb{Z}) \times H_1(X,\partial X;\mathbb{Z})\rightarrow \mathbb{Z}$$
such that$ \langle[S_k],[\widetilde{\gamma_n}\mid_{I_{t_j}}]\rangle\neq\langle[S_k],[\widetilde{\gamma_n}\mid_{I_{t_i}}]\rangle.$

Then $[\widetilde{\gamma_n}\mid_{I_{t_i}}]\neq[\widetilde{\gamma_n}\mid_{I_{t_j}}],$ meaning that $\widetilde{\gamma_n}\mid_{I_{t_i}}$ is not ambient isotopic to  $\widetilde{\gamma_n}\mid_{I_{t_j}}$ in $(X,\partial X).$
\end{proof}

The previous fact implies that the isotopy classes of $\widehat\gamma$-arcs in $T^1P$ corresponding to $\{\widetilde{\gamma_n}\mid_{I_{t_j}}\}^n_{j=1}$ are different. So by Corollary  \ref{1.2}, we conclude that\string:
$$\frac{v_3}{2} n < \Vol(M_{\widetilde{\gamma_n}}).$$
Moreover, by Theorem \ref{2}, if we take the canonical lifts of the sequence $\{\gamma_n\},$ then there exists a constant $V_0>0$ such that $\Vol(M_{\widehat{\gamma_n}})< V_0$ for every $n\in\mathbb{N}.$
\vskip .2cm
To finish the case $\Sigma=\Sigma_{1,2},$ it is enough to estimate the length of $\gamma_n$  by comparing it to its word length,
$$ \ell_{X} (\gamma_n)\leq (4+ n+ 2\ln (n))\ell_{max}\leq 5n\ell_{max},$$
where $\ell_{max}$ is the maximal length of three simple closed curves representing $a_1,t_1$ and $b_2$ with the same base point.
\vskip .2cm
We have proved the Proposition \ref{lin} for the case $\Sigma=\Sigma_{1,2}.$ Suppose now that $\Sigma$ contains a subsurface of topological type as $\Sigma_{1,2}.$ Proceeding as in Theorem \ref{pi-b-2} we get the wanted sequence of filling closed geodesics on $\Sigma,$ and their respective canonical lifts are constructed in the same way as in the previous case. Furthermore, if $\Sigma$ does not contain a subsurface of topological type as $\Sigma_{1,2},$ we consider $\pi:\Sigma'\rightarrow \Sigma$ a finite cover where $\Sigma'$ contains a subsurface $\Sigma_{1,2},$ and the induced covering $\widehat\pi$ from $PT^1\Sigma'$ to $PT^1\Sigma.$ Consider the previous geodesics $\{\gamma_n\}$  on $\Sigma'$ and $\{\pi(\gamma_n)\}$ in $\Sigma.$ Then the volume of the complement of $\widehat\pi(\widetilde\gamma_n)$ on $PT^1\Sigma$ is a constant (given by the degree of the covering) times the volume of the complement of the non-canonical lifts of all the translates of $\gamma_n$ under the group of deck transformations on $PT^1\Sigma',$  which is grater than the volume corresponding to the complement of $\widetilde\gamma_n$ on $PT^1\Sigma'.$ So the filling closed geodesics $\{\pi(\gamma_n)\}$ in $\Sigma$ and non-canonical lifts $\widehat\pi(\widetilde\gamma_n)$ on $PT^1\Sigma$ are the sequences wanted.
\end{proof}

One of the question that arises from the fact that the topology of $M_{\widetilde{\gamma}}$ does not depend on the metric given on $\Sigma$ is to find a metric for each element in the sequence $\{\gamma_n\}$ of Proposition \ref{lin} (see also Theorem \ref{pi-b-2}) which minimizes its length (\cite{Ker83}, Lemma 3.1). This implies that the volume lower bounds (Corollary \ref{1.2}) in terms of the length with respect this particular metric will be larger.  For example, in the following result, we choose for each $\gamma_n$ on Proposition \ref{lin} a different hyperbolic metric on $\Sigma_{1,2}$ which can be interpreted as an approximation of the minimal metric associated to each $\gamma_n,$ and rewrite the lower bound in terms of this new metrics.
 
\begin{proposition}\label{pi-b2}
 There exist  $X_n$ hyperbolic metrics on $\Sigma_{1,2}$ corresponding to the closed geodesics $\gamma_n$ of Proposition \ref{lin} with $\ell_{X_n}(\gamma_n)\nearrow \infty,$ such that\string:
$$\frac{v_3}{2\sqrt{2}}e^{\frac{\sqrt{\ell_{X_n}(\gamma_n)}}{2} } \leq  \Vol(M_{\widetilde{\gamma_n}}),$$
where $v_3$ is the volume of a regular ideal tetrahedron and $\widetilde{\gamma_n}$ are the sequence of lifts in Proposition \ref{lin}.
 \end{proposition}
 
\begin{proof}[\bf{Proof}] 
For each $n\in\mathbb{N}$ we construct a hyperbolic metric $X_n$ on $\Sigma_{1,2}$ by using the pants decomposition $\Pi$ of Proposition \ref{lin}. Let us fix the length of the boundary components of the pair of pants corresponding to the piece $\Sigma_{1,1}$ by $4\ln (\sqrt{2}n),$ this implies that the length of the seam arcs joining the boundary components is approximately $\frac{1}{n}.$ Also fix the length of the boundary components of $\Sigma_{1,2}$ by a constant. 
 \vskip .2cm
The conclusion follows from the following inequality, which results from the upper estimation of the geodesic length with respect the hexagonal decomposition induced by $\Pi$ and the seam edges\string:
$$ \ell_{X_n} (\gamma_n)\leq 4{\ln(\sqrt{2}n)}^2.$$ 
\end{proof}

\vskip .01cm
{ \SMALL I.R.M.A.R., UNIVERSIT\'E DE RENNES I, CAMPUS DE BEAULIEU, 35042 RENNES CEDEX, FRANCE}\\
{ \SMALL \textit{E-mail address:} \textbf{jose.rodriguezmigueles@etudiant.univ-rennes1.fr}}

\begin{thebibliography}{99}

\bibitem{Ada88}
C. Adams,\emph{Volumes of N-Cusped Hyperbolic 3-Manifolds}, J. London Math. Soc. (2) 38 (1988), no. 3, 555-565.

\bibitem{AST07}
I. Agol, P. Storm  and W. Thurston,  (2007). \emph{Lower bounds on volumes of hyperbolic Haken 3-manifolds.} J. Amer. Math. Soc. 20 No.4, 1053-1077.

\bibitem{BPS16}
M. Bergeron, T. Pinsky and L. Silberman, \emph{An Upper Bound for the Volumes of Complements of Periodic Geodesics}, International Mathematics Research Notices, Vol. 2017, No. 00, pp. 1-23.

\bibitem{BK09}
J. Birman and I. Kofman, \emph{A new twist on Lorenz links}, J. Topol. 2 (2009) no. 2, p. 227-248.

\bibitem{BWS83}
J. Birman and R. Williams,  \emph{Knotted periodic orbits in dynamical systems -II: Knot holders for fibered knots}, Cont. Math 20, 1-60 (1983).

\bibitem{BPS17}
A. Brandts, T. Pinsky and L. Silberman, \emph{Volumes of hyperbolic three-manifolds associated to modular links}, Symmetry (2019), 11(10), 1206.

\bibitem{CB88}
A. Casson and S. Bleiler, \emph{Automorphisms of surfaces after Nielsen and Thurston}. Number 9. Cambridge University Press, 1988.

\bibitem{Cha10}
M. Chas, \emph{Minimal intersection of curves on surfaces}, Geom. Dedicata 144 (2010) 25-60 MR.

\bibitem{CKP16}
 A. Champanerkar, I. Kofman and J. Purcell, \emph{Volume bounds for weaving knots}
Algebraic and Geoemtric Topology, Vol. 16 (2016), No. 6, 3301-3323.

\bibitem{Deh11}
P. Dehornoy,
 \emph{Les noeuds de Lorenz}, L'Enseignement Math\'ematique (2) 57 (2011), 211-280.
 
 \bibitem{Duk88}
  W. Duke,  \emph{Hyperbolic distribution problems and half-integral weight Maass forms.} Invent. Math., 92(1):73-90, 1988.
 
\bibitem{FM12}
B. Farb and D. Margalit, \emph{A primer on mapping class groups}, volume 49 of Princeton Mathe-
matical Series. Princeton University Press, Princeton, NJ, 2012.

\bibitem{FH13}
P. Foulon and B. Hasselblatt, \emph{Contact Anosov flows on hyperbolic 3-manifolds}, Geom. Topol. 17 (2013)
1225-1252.

\bibitem{Ghy07}
E. Ghys, \emph{Knots and dynamics}, International Congress of Mathematicians. Vol. I, Eur. Math. Soc., Z\"{u}rich, 2007, pp. 247-277. MR 2334193 (2008k:37001).

\bibitem{GS97} 
M. de Graaf and A. Schrijver, \emph{ Making curves minimally crossing by Reidemeister moves.}
J. Combin. Theory Ser. B , 70(1):134-156, 1997.

\bibitem{Gro82}
 M. Gromov,  \emph{Volume and bounded cohomology}. Inst. Hautes \'Etudes Sci. Publ. Math. 56
(1982), 5-99. Zbl 0516.53046 MR 0686042

\bibitem{HS85}
J. Hass and P. Scott,\emph{ Intersections of curves on surfaces.} Israel J. Math. , 51(1-2):90-120, 1985.

\bibitem{HS94}
J. Hass and P. Scott, \emph{Shortening curves on surfaces.} Topology, 33(1):25-43, 1994.

\bibitem{He76}
J. Hempel, \emph{3-Manifolds,}  Ann. of Math. Studies, No. 86. Princeton University Press, Princeton, N. J. (1976).

\bibitem{IM02}
K. Ichihara et K. Motegi,   \emph{Sectional knots in Seifert fibered 3-manifolds}. Hyperbolic spaces and discrete groups, II (Japanese) (Kyoto, 2001). Srikaisekikenkysho Kkyroku No. 1270 (2002), 101-111. 57M25

\bibitem{JS178}
W. H. Jaco and P. B. Shalen, \emph{A new decomposition theorem for irreducible sufficiently-large 3-manifolds}, 
Algebraic and geometric topology, 1978, 71-84.

\bibitem{JS79}W. H. Jaco and P. B. Shalen, \emph{Seifert fibered spaces in 3-manifolds}, vol. 21, Memoirs of the American Mathematical Society, no. 220, AMS Chelsea Publishing, 1979.

\bibitem{Joh79}
K. Johannson,  \emph{Homotopy equivalences of 3-manifolds with boundaries,} Lecture Notes in Mathematics, vol. 761, Springer, Berlin, 1979.

\bibitem{Ker83}
S. Kerckhoff,  \emph{The Nielsen realization problem}, Ann. of Math. (2) 117 (1983).

\bibitem{Os06} J. Osoinach, \emph{Manifolds obtained by surgery on an infinite number of knots in S3} ,Topology
45, no. 4 (2006): 725-33.

\bibitem{Ot79}
J. P. Otal, \emph{Thurston's hyperbolization of Haken manifolds,}  In Surveys in differential geometry, Vol. III (Cambridge, MA, 1996), pages 77-194. Int. Press, Boston, MA, 1998.

\bibitem{Ser85}
C. Series, \emph{The modular surface and continued fractions,}  J. London Math. Soc. (2) 31 (1985), no. 1, 69-80. MR 810563.

\bibitem{Ser77}
J.P. Serre, \emph{Arbres, amalgames, SL2}. Soci\'et\'e Math\'ematique de France, Paris, 1977. R\'edig\'e avec la collaboration de Hyman Bass, Ast\'erisque, No. 46.

\bibitem{Thu79}W.P. Thurston, \emph{The geometry and topology of three-manifolds}, lecture notes, Princeton University, 1979.

\bibitem{Thu82}W.P. Thurston, \emph{Three-dimensional manifolds, Klenian groups and hyperbolic geometry,} Bull. Amer. Math. Soc. 6, 1982.


\end{thebibliography}
\end{document}